\documentclass[12pt,oneside,reqno]{amsart}
\usepackage{latexsym,amsmath,amsfonts,amsthm}
\usepackage[usenames]{color}
\usepackage[usenames,dvipsnames,svgnames,table]{xcolor}
\usepackage[all]{xy}  

\usepackage{multicol}

\usepackage{upgreek}

\usepackage[]{graphicx}
\usepackage{wrapfig}

\DeclareSymbolFont{rmlargesymbols}{OMX}{mdbch}{m}{n}
\DeclareMathSymbol{\rmintop}{\mathop}{rmlargesymbols}{82}
\DeclareMathSymbol{\rmointop}{\mathop}{rmlargesymbols}{72}
\DeclareMathSymbol{\rmsumop}{\mathop}{rmlargesymbols}{80}
\DeclareMathSymbol{\rmunionop}{\mathop}{rmlargesymbols}{83}
\DeclareMathSymbol{\rmintersectop}{\mathop}{rmlargesymbols}{84}
\DeclareMathSymbol{\rmtensorop}{\mathop}{rmlargesymbols}{79}
\DeclareMathSymbol{\rmdirectsumop}{\mathop}{rmlargesymbols}{77}

\usepackage{mathrsfs}

\usepackage{mathptmx}
\usepackage{bookman}
\usepackage{relsize} 

\usepackage[font=small,labelfont=bf]{caption}
\usepackage{sidecap}
\usepackage{float}

\input xy
\xyoption{all}

\usepackage{amsmath}
\usepackage{amsfonts}
\usepackage{hyperref}
\usepackage{amssymb}
\usepackage[]{graphicx}
\usepackage{upgreek}
\usepackage{mathptmx}

\usepackage{amssymb}
\usepackage{amscd}
\usepackage{array}
\usepackage{verbatim}

\theoremstyle{definition}
\newtheorem{thm}{Theorem}[section]
\newtheorem{lem}[thm]{Lemma}

\newtheorem{dfn}[thm]{Definition}

\newtheorem{rmk}[thm]{Remark}
\newtheorem{ntn}[thm]{Notation}

\newtheorem{std}[thm]{Standing Condition}

\newcommand{\QED}{\rule{0.4em}{2ex}}

\def\conj#1{\overline{#1}}
\def\inner#1#2#3{ \langle #1,\, #2 \rangle_{_{#3}}}
\def\Dinner#1#2{ \langle #1,\, #2 \rangle_{{}_D}}
\def\Dpinner#1#2{ \langle #1,\, #2 \rangle_{{}_{D^\perp}}}

\def\vartau{\uptau}

\def\ft#1{\widehat{#1}}
\def\T{\bold T}

\def\ccite#1{\textcolor{Red}{\cite{#1}}}

\numberwithin{equation}{section}

\begin{document}

\title[ K-T\lowercase{heory} \lowercase{of} A\lowercase{pproximately} C\lowercase{entral} P\lowercase{owers}-R\lowercase{ieffel} P\lowercase{rojections}, \ S. W\lowercase{alters} ]
{\Large \rm K-T\lowercase{heory} \lowercase{of} A\lowercase{pproximately} \\ 
\vskip5pt  C\lowercase{entral} P\lowercase{rojections} \lowercase{in} \lowercase{the} F\lowercase{lip} O\lowercase{rbifold}}

\author{S\lowercase{amuel} G.\,W\lowercase{alters} \\ 
{\Tiny U\lowercase{niversity of} N\lowercase{orthern} B\lowercase{ritish} 
C\lowercase{olumbia}}}
\date{Year of the Pandemic, 2020} 
\address{Department of Mathematics and Statistics, University  of Northern B.C., Prince George, B.C. V2N 4Z9, Canada.}
\email[]{\ samuel.walters@unbc.ca \  \text{or} \ walters@unbc.ca}
\urladdr{http://web.unbc.ca/~walters/}

\begin{abstract}
For an approximately central (AC) Powers-Rieffel projection $e$ in the irrational Flip orbifold C*-algebra $A_\theta^\Phi,$ where $\Phi$ is the Flip automorphism of the rotation C*-algebra $A_\theta,$ we compute the Connes-Chern character of the cutdown of any projection by $e$ in terms of K-theoretic invariants of these projections.  This result is then applied to computing a complete K-theoretic invariant for the projection $e$ with respect to central equivalence (within the orbifold). Thus, in addition to the canonical trace, there is a $4\times6$ K-matrix invariant $K(e)$ arising from unbounded traces of the cutdowns of a canonically constructed basis for $K_0(A_\theta^\Phi) = \mathbb Z^6$. Thanks to a theorem of Kishimoto, this enables us to tell when AC projections in $A_\theta^\Phi$ are Murray-von Neumann equivalent via an approximately central partial isometry (or unitary) in $A_\theta^\Phi$. As additional application, we obtain the K-matrix of canonical SL$(2,\mathbb Z)$-automorphisms of $e$ and show that there is a subsequence of $e$ such that $e, \sigma(e),  \kappa(e), \kappa^2(e), \sigma\kappa(e), \sigma\kappa^2(e)$ -- which are the orbit elements of $e$ under the symmetric group $S_3 \subset$ SL$(2,\mathbb Z)$ -- are pairwise centrally not equivalent, and that each SL$(2,\mathbb Z)$ image of $e$ is centrally equivalent to one of these, where $\sigma, \kappa$ are the Fourier and Cubic transform automorphisms of the rotation algebra.
\end{abstract}

\maketitle

\begin{multicols}{2}
  { \tiny{\tableofcontents}}
\end{multicols}

\newpage

\textcolor{blue}{\Large \section{Introduction}}

We study the problem of when a pair of approximately central (AC) projections are Murray-von Neumann equivalent by means of a  partial isometry (or unitary) that is approximately central. In doing so, additional K-theoretic information on the AC projections is required.
Such information was found by Kishimoto \ccite{AK}\footnote{Kishimoto's main Theorem 2.1 in \ccite{AK} is stated for separable, nuclear, purely infinite, simple C*-algebras satisfying UCT. In Remark 2.9 of \ccite{AK}, he notes that it also applies to simple AT C*-algebras of real rank zero, which includes the irrational rotation C*-algebras, known to be AT from \ccite{EE} and \ccite{EL}, and also includes their canonical orbifolds under the canonical automorphisms of order 2, 3, 4, and 6 as they are known to be AF from \ccite{BK} \ccite{ELPW} \ccite{SW-CMP} \ccite{SWcrelles}. Recall that AF-algebras are AT-algebras (\ccite{Rordam}, Corollary 3.2.17).} (Theorem 2.1) for certain classes of C*-algebras, which include algebras studied in this paper.  The objective of this paper is to formalize this information into a topological K-theory invariant for AC projections, and proceed with computing it explicitly for an AC Flip-invariant Powers-Rieffel projection in the irrational rotation C*-algebra $A_\theta,$ which will be denoted throughout this paper by $e$ (see equation \eqref{ezetaE}).  The projection $e$ is similar to one constructed by Elliott and Lin \ccite{EL}, and is essentially the same as the unit projection in the Elliott-Evans tower construction \ccite{EE}.

\medskip

Our main results are stated in Theorems \ref{MaintheoremA}, \ref{Kmatrixthm}, and \ref{S3imagesofe} of this section.
\medskip

In this paper we will be concerned with the Flip orbifold C*-algebra
\[
A_\theta^\Phi = \{x\in A_\theta: \Phi(x)=x\}
\]
the fixed point C*-subalgebra of the irrational rotation algebra $A_\theta$ under the Flip automorphism $\Phi$ defined by 
\[
\Phi(U) = U^{-1}, \qquad \Phi(V) = V^{-1}
\]
where $U,V$ are unitaries generating $A_\theta$ (also called noncommutative torus)  satisfying the usual Heisenberg commutation relation
\begin{equation}\label{VUUV}
VU= e^{2\pi i\theta} UV.
\end{equation}
Throughout the paper, $\theta$ is a fixed irrational number, $0 < \theta < 1$. Both $A_\theta$ and its Flip orbifold have canonical bounded traces which are unique normalized traces denote by $\tau$. 

\medskip

The approximately central projections studied in this paper depend on integer parameters. In our case for example, $e = e_{q',q,p,\theta}$ is a Powers-Rieffel projection that depends on a sequence of consecutive convergents $p/q,\ p'/q'$ of $\theta$. Since the rotation algebra $A_\theta$ is generated by the unitaries $U,V$, a projection $e$ is AC in $A_\theta$ if $\|e U - U e\|, \|e V - V e\| \to0$  as $q\to \infty$. A projection $e$ is AC in the Flip orbifold $A_\theta^\Phi$ if
\[
\|e(U+U^*) - (U+U^*) e\| \to0,\quad \|e (V+V^*) - (V+V^*)e\| \to0
\]
as $q\to \infty$ (since it is known from \ccite{BEEKa} that $U+U^*, V+V^*$ generate $A_\theta^\Phi$). We do not know if AC in $A_\theta^\Phi$ implies AC in $A_\theta$ in general (though for many projections in the C*-algebra generated by certain powers of $U,V$ this can be checked).
\bigskip

\begin{dfn} Two AC projections are {\it centrally equivalent} in an algebra $A$ (or {\it AC-equivalent} in $A$) if they are Murray-von Neumann equivalent by a partial isometry in $A$ that is approximately central in $A$ (for large enough parameter).
\end{dfn}
\medskip

Kishimoto's Theorem 2.1 in \ccite{AK} (restated in Section 2.4 below),  
as applied to the Flip orbifold $A_\theta^\Phi$ (which known to be an AF-algebra \ccite{BK}, \ccite{SW-CMP}), implies that two AC projections $e$ and $f$ in $A_\theta^\Phi$ are centrally equivalent in $A_\theta^\Phi$ if and only if the cutdown of a given finite generating set of projections $[P_j]$ for $K_0(A_\theta^\Phi)$ by $e$ and $f$ have the same $K_0$-class,
\begin{equation}\label{efPj}
[\chi(eP_je)] = [\chi(fP_jf)] \ \ \in \ K_0(A_\theta^\Phi)
\end{equation}
for each $j,$ where $\chi$ is the characteristic function of the interval $[\frac12,\infty)$. Of course, equation \eqref{efPj} is understood to hold for large enough integer parameters which $e$ and $f$ depend on.
Since $A_\theta^\Phi$ is AF (so its $K_1 = 0$), the $K_1$ side of Kishimoto's conditions (see Theorem 2.2 below) are trivially satisfied. In our particular case, $P_j$ are projections in $A_\theta^\Phi$.\footnote{It seems reasonable to expect that if $[\chi(ege)] < [\chi(fgf)]$ for each $g = P_j$, then there exist AC partial isometry $u$ such that $uu^* = e$ and $u^*u \le f$; but the author has no proof. }

\medskip

In Section 2.3 we construct a specific basis $[P_1], \dots, [P_6]$ for $K_0(A_\theta^\Phi) = \mathbb Z^6$ (see \eqref{basis} and \eqref{Kbasis}), with specific projections $P_s$ in $A_\theta^\Phi,$ with respect to which we compute the classes $[\chi(eP_se)]$ -- which would therefore determine the central equivalence class of the  projection $e$ in the Flip orbifold.  These $K_0$-classes will be identified explicitly by computing their Connes-Chern character  
\begin{equation} 
\T : K_0(A_\theta^\Phi) \to \mathbb R^5, \qquad
\T(x) = (\vartau(x); \phi_{00}(x), \phi_{01}(x), \phi_{10}(x), \phi_{11}(x))
\end{equation}
in terms of the canonical trace $\tau,$ and four basic unbounded traces $\phi_{jk}$ (defined in Section 2.1 below). The map $\T$ is known to be a group monomorphism for irrational $\theta$ (see \ccite{SWa}, Proposition 3.2). Therefore, in terms of the Connes-Chern character we will calculate the numerical invariants for $e$
\[
\tau \chi(eP_se), \qquad \phi_{jk} \chi(eP_se) 
\]
($s=1,...,6,\ jk=00,01,10,11$ ). For ease of notation, let us write
\[
\chi_s := \chi(e P_s(\theta) e)
\]
for the cutdown projections by $e$ (for large enough parameter).  We also find it convenient to organize these classes for $e$ into a vector
\[
\vec \tau(e) =  \begin{bmatrix} 
\tau \chi_1 & \tau \chi_2 & \tau \chi_3 & \tau \chi_4 & \tau \chi_5& \tau \chi_6
\end{bmatrix}  
\]
consisting of the canonical traces of the cutdowns, together with a topological K-matrix involving the unbounded traces which we denote by
\[
K(e) = 
\begin{bmatrix} 
\phi_{00}(\chi_1) & \phi_{00}(\chi_2) & \phi_{00}(\chi_3) & \phi_{00}(\chi_4) & 
\phi_{00}(\chi_5) & \phi_{00}(\chi_6)
\\
\phi_{01}(\chi_1) & \phi_{01}(\chi_2) & \phi_{01}(\chi_3) & \phi_{01}(\chi_4) & 
\phi_{01}(\chi_5) & \phi_{01}(\chi_6)
\\
\phi_{10}(\chi_1) & \phi_{10}(\chi_2) & \phi_{10}(\chi_3) & \phi_{10}(\chi_4) & 
\phi_{10}(\chi_5) & \phi_{10}(\chi_6)
\\
\phi_{11}(\chi_1) & \phi_{11}(\chi_2) & \phi_{11}(\chi_3) & \phi_{11}(\chi_4) & 
\phi_{11}(\chi_5) & \phi_{11}(\chi_6)
\end{bmatrix}
\]
wherein the $(jk,s)$-entry consists of the unbounded trace $\phi_{jk}(\chi_s)$. Therefore, in the Flip orbifold, the central equivalence class of $e$ is fully determined by the pair consisting of the canonical trace vector $\vec \tau(e),$ and the $4\times6$ topological matrix $K(e)$.
\medskip

\begin{ntn}\label{notation}
We shall use the divisor delta function $\delta_n^m = 1$ if $n$ divides $m$, and $\delta_n^m = 0$ otherwise. We also use the notation $e(t) := e^{2\pi it}$. Thus, $\rmsumop_{j=0}^{n-1} e(\tfrac{mj}n) = n \delta_n^m$.
\end{ntn}

\begin{std}\label{standing} Without loss of generality we can assume that there are infinitely many consecutive convergents $\tfrac{p}q, \tfrac{p'}{q'}$ of $\theta$ such that
\[
\frac{p}q < \theta < \frac{p'}{q'}, \qquad \frac12 < q'(q\theta - p)  < \frac45.
\]
(See Remark \ref{tracecondition} for why.) This will be assumed in the hypotheses of Theorems \ref{MaintheoremA}, \ref{Kmatrixthm}, and \ref{S3imagesofe}.
\end{std}
\medskip

The AC Flip-invariant Powers-Rieffel projection in $A_\theta^\Phi$ that will be studied throughout this paper is 
\begin{align}\label{ezetaE}
e \ := \ e_{q',q,p,\theta} \ &=\ G_\tau(U^{q'}) V^{-q} + F_\tau(U^{q'}) + V^q G_\tau(U^{q'})
\\
&\ = \  \zeta_{q',q,p,\theta} \ \mathcal E(q'(q\theta-p))		\label{ezeta}
\end{align}
depending on the convergent parameters $p,q,p',q',$ as stated in the Standing Condition, with trace $\tau(e) = q'(q\theta - p) \in(\tfrac12,\tfrac45)$. The Rieffel functions $F_\tau, G_\tau$ in \eqref{ezetaE} and the continuous field $\mathcal E(t)$ in \eqref{ezeta} are described in Section 2.2 (see \eqref{rieffelproj}), and the C*-morphism $\zeta$ is defined in \eqref{zetamorphism}.  No confusion should arise with occasionally denoting the trace of $e$ simply by $\tau := q'(q\theta - p)$.
\medskip

In light of this background, our results can now be stated in terms of the following three theorems (some of the notation of which is explained later).
\medskip

\begin{thm}\label{MaintheoremA}
Let $\theta$ be irrational with convergents satisfying the Standing Condition. Let $t \to P(t)$ be a continuous section, defined in some neighborhood of $\theta,$ of smooth projections of the continuous field $\{ A_t^\Phi \}_{0<t<1}$ of Flip orbifolds.  Then the unbounded trace of the cutdown projection $\chi(e P e),$ where $P=P(\theta),$ is given by
\[
\phi_{jk} \chi(e P e) = a_{jk}^- C_0(P) + a_{jk}^+ C_1(P)
\]
for $jk=00,01,10,11,$ where
\begin{equation}\label{CP}
C_0(P) = \phi_{00}(P) + \phi_{10}(P), \qquad
C_1(P) = \phi_{0,q'}(P) + (-1)^{p'} \phi_{1,q'}(P)
\end{equation}
and
\[
a_{jk}^- =
\begin{cases} \tfrac12 (-1)^{pjk},  &\text{    for even } q' \\ \\
\updelta_2^{q} \updelta_2^{j} \updelta_2^{k}  + \tfrac12 (-1)^{pjk}\updelta_2^{q-1}   &\text{     for odd } q'
\end{cases}
\]
\smallskip
\[
a_{jk}^+ =
\begin{cases} 
 \tfrac12 (-1)^{j+pjk} &\text{    for even } q' \\ \\
 \tfrac12 (-1)^{p'j} \left[ \updelta_2^{k-1} + (-1)^{pj} \updelta_2^{q-k-1}  \right]    &\text{     for odd } q'.
\end{cases}
\]
The canonical trace of the cutdown is 
\[
\tau \chi(e Pe) = \tau(e) \tau_{p'/q'}(P(\tfrac{p'}{q'}))
\]
where $\tau_{p'/q'}$ is the canonical normalized trace of the rational rotation algebra $A_{p'/q'}$.
\end{thm}

\medskip

With this result established, we can write down the topological K-matrix of the projection $e$ according to the following theorem.

\medskip

\begin{thm}\label{Kmatrixthm}
Let $\theta$ be an irrational number. The Powers-Rieffel projection $e$ has K-matrix given by the $4\times 6$ matrices according to parities (as indicated by subscripts):
\begin{equation*} 
K(e_{q',q,p,\theta}) \ = \ \frac12 
\begin{bmatrix} 
2 & 0 & 1 & 1 & 1 & 1 
\\ 
2 & 0 & 1 & 1 & 1 & 1 
\\ 
0 & 0 & 1 & -1 & 1 & -1
\\ 
0 & 0 & (-1)^p & (-1)^{p+1} & (-1)^p & (-1)^{p+1} 
\end{bmatrix}_{q' \text{ even}}		
\end{equation*}
\begin{equation*}
K(e_{q',q,p,\theta}) = 
\begin{bmatrix} 
1 & 0 & 1 & 0 & 1 & 0 
\\ 
0 & 0 & \updelta_2^{p'} & \updelta_2^{p'-1} & -\updelta_2^{p'} & -\updelta_2^{p'-1}
\\ 
0 & 0 & 0 & 0 & 0 & 0
\\ 
0 & 0 & 0 & 0 & 0 & 0
\end{bmatrix}_{q \text{ even}}
\end{equation*}

\begin{equation*}
K(e_{q',q,p,\theta}) =
\frac12 \begin{bmatrix} 
1 & 0 & 1+\updelta_2^{p'} & \updelta_2^{p'-1} & \updelta_2^{p'-1} & -\updelta_2^{p'-1} 
\\ 
1 & 0 & 1+\updelta_2^{p'} & \updelta_2^{p'-1} & \updelta_2^{p'-1} & -\updelta_2^{p'-1} 
\\ 
1 & 0 & \updelta_2^{p'-1} &  -\updelta_2^{p'-1} & 1+\updelta_2^{p'} & \updelta_2^{p'-1}
\\ 
(-1)^p & 0 & \updelta_2^{p'-1} & -\updelta_2^{p'-1} & (-1)^p(1 + \updelta_2^{p'}) & 
\updelta_2^{p'-1} 
\end{bmatrix}_{q,q' \text{ both odd}}.
\end{equation*}
In addition, the canonical trace vector is
\[
\vec \tau(e_{q',q,p,\theta}) = 
(q\theta-p) \begin{bmatrix} q' & \ p_2 & \ p' & \ p' & \ p' & \ p' \end{bmatrix} 
\]
where
\[
p_2 := 
 \begin{cases} 
2p' &\text{for } 0 < \theta < \tfrac12
\\
2(q' - p') &\text{for } \tfrac12 < \theta < 1.
\end{cases}
\]
 \end{thm}

\medskip

\noindent{\bf Application.} One interesting application of the preceding theorem is determination of whether or not $\alpha(e)$ and $\beta(e)$ are centrally equivalent for smooth automorphisms $\alpha, \beta$ of the Flip orbifold. Or, equivalently, when is $\alpha(e)$ centrally equivalent to $e$? We will answer this for the canonical automorphisms given by the Fourier $\sigma$ and Cubic $\kappa$ transforms (studied in \ccite{FB} \ccite{BW}  \ccite{SWChern} \ccite{SWcjm} \ccite{SWcrelles}) defined by
\begin{align*}
&  &  & & \sigma(U) &= V^{-1},  &	\sigma(V) &= U		&  &  & & 
\\
&  &  & & \kappa(U) &= e(-\tfrac\theta2)U^{-1}V,  & \kappa(V) &= U^{-1}.	&  &  & & 
\end{align*}
These have order 2 and 3, respectively, on the Flip orbifold, and they have a simple relations with the unbounded traces: $\sigma$ swaps $\phi_{01}$ and $\phi_{10},$ fixing $\phi_{00}, \phi_{11},$ while $\kappa$ induces the cyclic permutation $\phi_{11} \to \phi_{01} \to \phi_{10},$ and fixing $\phi_{00}$. With this information at hand, one easily calculates the action of these automorphisms on the basis $[P_1], \dots, [P_6]$ of $K_0(A_\theta^\Phi)$ mentioned earlier. This would then lead to the following result.

\medskip

\begin{thm}\label{S3imagesofe}
Considering the Powers-Rieffel AC projection $e$ (given by \eqref{ezetaE}) as running over parameters where $q$ is even and $p'$ is odd, one has its K-matrix and those of its Fourier and Cubic transforms
\[
K(e) = \begin{bmatrix}
1 & 0 & 1 & 0 & 1  & 0
\\
0 & 0 & 0 & 1 & 0 &  -1
\\
0  & 0 & 0 & 0 & 0 & 0
\\
0  & 0 & 0 & 0 & 0 & 0
\end{bmatrix},
\qquad
K(\kappa(e)) = \begin{bmatrix}
1 & 0 & 1 & 1 & 0  & 0
\\
0  & 0 & 0 & 0 & 0 & 0
\\
0 & 0 & 0 & 0 & -1 & 1
\\
0  & 0 & 0 & 0 & 0 & 0
\end{bmatrix}
\]
\[
K(\kappa^2(e)) = \begin{bmatrix}
1 & 0 & 1 & 0 & 0  & 1
\\
0 & 0 & 0 & -1 & 1 & 0
\\
0  & 0 & 0 & 0 & 0 & 0
\\
0  & 0 & 0 & 0 & 0 & 0
\end{bmatrix},
\qquad
K(\sigma(e)) = \begin{bmatrix}
1 & 0 & 1 & 1 & 0  & 0
\\
0  & 0 & 0 & 0 & 0 & 0
\\
0 & 0 & 0 & 0 & 1 &  -1
\\
0  & 0 & 0 & 0 & 0 & 0
\end{bmatrix}
\]
\[
K(\sigma\kappa(e)) = \begin{bmatrix}
1 & 0 & 1 & 0 & 1 & 0
\\
0 & 0 & 0 & -1 & 0 & 1
\\
0  & 0 & 0 & 0 & 0 & 0
\\
0  & 0 & 0 & 0 & 0 & 0
\end{bmatrix},
\qquad
K(\sigma\kappa^2(e)) = \begin{bmatrix}
1 & 0 & 1 & 0 & 0  & 1
\\
0  & 0 & 0 & 0 & 0 & 0
\\
0 & 0 & 0 & 1 & -1 & 0
\\
0  & 0 & 0 & 0 & 0 & 0
\end{bmatrix}.
\]
In particular, as these matrices are pairwise distinct, the AC projections
\begin{equation}\label{ecyclic}
e, \quad \sigma(e), \quad \kappa(e), \quad \kappa^2(e), \quad \sigma\kappa(e), \quad  \sigma\kappa^2(e)
\end{equation}
are pairwise not centrally equivalent. 
\end{thm}

\medskip

Note that the six projections in \eqref{ecyclic} all have the same orbifold $K_0$ class and constitute the $S_3$ orbit of $e$ (where the symmetric group $S_3$ is generated by $\sigma$ and $\kappa$). Indeed, in the case considered in the preceding theorem where $q$ is even ($q',p$ odd), the Connes-Chern character of $e$ can be obtained from \eqref{etraces} as
\[
\T(e) = (\tau(e);\ 1, 0, 0, 0)
\]
where the trace of $e$ is also the trace of any automorphism of $e$ (by uniqueness of the normalised canonical trace). In view of what was just noted regarding how $\sigma, \kappa$ act on the $\phi_{jk},$ we see that all the projections in \eqref{ecyclic} have exactly the same Connes-Chern character as $e,$ so that they are all Murray-von Neumann equivalent (in the Flip orbifold). Theorem \ref{S3imagesofe}, however, says that they are pairwise centrally inequivalent.
\medskip

In \ccite{WaltersModular} we showed that in fact for any AC projection $e$ in the Flip orbifold, its $S_3$ orbit \eqref{ecyclic} gives the only possible central classes, in the sense that for any canonical automorphism $\alpha$ arising from SL$(2,\mathbb Z),$ the projection $\alpha(e)$ is centrally equivalent to one of the projections in \eqref{ecyclic}.  The point of Theorem \ref{S3imagesofe} is then that all six projections can be pairwise centrally distinct. It is possible also to give examples where some or all of the projections in \eqref{ecyclic} are centrally equivalent - see Theorem 1.4 of  \ccite{WaltersModular}.
\medskip

We have included many details in our calculations in this paper so as to save the reader from numerous and onerous checking (and hopefully to also be clear about our reasoning). We hope this may be of help. 

\medskip

\noindent{\bf Structure Of Paper.} Let's summarize what we do in this paper.

In Section 2 we gather the necessary background material, notation, results and introduce our notation for the K-matrix relative to a constructed basis for group $K_0(A_\theta^\Phi)$ required by our proofs.

In Section 3 the approximately central Powers-Rieffel projection $e$ is constructed from the continuous field constructed in Section 2 and its unbounded traces are calculated.

In Section 4, the projection is realized as a C*-inner product $e = \Dinner{f}{f}$ using Riefflel's equivalence bimodule paradigm \ccite{MRb}. This realization will facilitate computation of the topological invariants of the cutdown $\chi(ePe)$ of any projection $P$ by $e,$ which is carried out in Section 5. This will prove Theorem \ref{MaintheoremA}. 

In Section 6, Theorem \ref{MaintheoremA} is applied to the $K_0$-basis constructed in Section 2.3 in order to obtain the full K-matrix of $e$ (for various parity situations of integer parameters that $e$ depends on). This will then prove Theorem \ref{Kmatrixthm}.

In Section 7, the unbounded traces of two technical C*-inner products required in Section 6  are computed (see Lemma \ref{phiV1V3}).

Sections 8 and 9 are appendices involving some basic unbounded trace computations used in the paper. (The Acknowledgement paragraph is right before the References.)

\medskip

In a forthcoming paper \ccite{WaltersKAC} we calculate the K-matrix of Fourier-invariant  projections in the Fourier orbifold $A_\theta^\sigma$ for the Fourier transform $\sigma$. The greater complication in the Fourier case arises from the fact that Fourier-invariant projections do not have a Powers-Rieffel form, from additional unbounded traces on $K_0(A_\theta^\sigma) = \mathbb Z^9,$ as well as additional $K_0$-basis elements, giving rise to a $9\times6$ K-matrix.

\textcolor{blue}{\Large \section{Background Material}}

In this section we write down the relevant Connes-Chern character for the Flip orbifold; define a continuous field of Powers-Rieffel projections, which is a continuous section of the continuous field of Flip orbifolds $\{A_t^\Phi: 0<t<1\};$ construct a canonical basis consisting of continuous fields of projections for $K_0(A_t^\Phi);$ state Kishomoto's Theorem (\ccite{AK}, Theorem 2.1) in the form we require; and state a couple of Poisson Summation formulas used in our calculations.

\bigskip

\subsection{The Connes-Chern Character} 
The Flip automorphism $\Phi$ of the rotation C*algebra $A_\theta$ has four associated unbounded $\Phi$-traces $\phi_{jk}^\theta$ defined on the basic unitaries $U^mV^n,$ satisfying \eqref{VUUV}, by (see \ccite{SWa} or \ccite{SW-CMP})
\begin{equation}
\phi_{jk}^\theta(U^mV^n)\ =\ e(-\tfrac{\theta}2 mn)\,\updelta_2^{m-j} \updelta_2^{n-k}
\end{equation}
for $jk = 00, 01,10,11$,  $m,n\in \mathbb Z$ (and $\updelta_a^{b}$ was defined in Notation \ref{notation}). Invariably, we may write $\phi_{jk}^\theta$ simply as $\phi_{jk}$ when $\theta$ is understood. These are linear functionals defined on the canonical smooth dense *-subalgebra $A_\theta^\infty$ satisfying the $\Phi$-trace property
\[
\phi_{jk}(xy) = \phi_{jk}(\Phi(y)x)
\]
for $x,y\in A_\theta^\infty$. (Of course, such maps are $\Phi$-invariant). In addition, they are Hermitian maps: they are real on Hermitian elements.  Clearly, on the smooth orbifold $A_\theta^{\Phi,\infty}$ -- the fixed point *-subalgebra of $A_\theta^\infty$ under the Flip -- they give rise to trace functionals that are not continuous in the C*-norm. Together with the canonical trace $\vartau$ one has the Connes-Chern character which we write as
\begin{equation} 
\T : K_0(A_\theta^\Phi) \to \mathbb R^5, \qquad
\T(x) = (\vartau(x);\, \phi_{00}(x),\, \phi_{01}(x),\, \phi_{10}(x),\, \phi_{11}(x)).
\end{equation}
In \ccite{SWa} (Proposition 3.2) this map is known to be injective for irrational $\theta$.\footnote{In \cite{SWa} we worked with the crossed product algebra $A_\theta \times_\Phi \mathbb Z_2$, but since this algebra is strongly Morita equivalent to the fixed point algebra, the injectivity follows for the latter and is easy to see.} For the identity element one has $\T(1) = (1; 1, 0, 0, 0)$. The ranges of the traces $\phi_{jk}$ on projections in $A_\theta^\Phi$ are known to be half-integers, while the canonical trace has range $(\mathbb Z + \mathbb Z\theta) \cap [0,1]$ on projections.

\bigskip

\subsection{Continuous field of Rieffel projections}

There is a natural continuous (section) field $\mathcal E: [\frac12,1) \to \{A_t\}$ of Flip-invariant Powers-Rieffel projections
\begin{equation}\label{rieffelproj}
\mathcal E(t) = G_t(U_t) V_t^{-1} + F_t(U_t) + V_t G_t(U_t) 
\end{equation}
where 
\[
F_t(x) = \rmsumop_{n \in\mathbb Z} f_t(x+n), \qquad 
G_t(x) = \rmsumop_{n \in\mathbb Z} g_t(x+n)
\]
are periodizations of the $t$-parameterized family ($\tfrac12 \le t < 1$) of continuous (or smooth) functions $f_t(x), g_t(x)$ on $\mathbb R$, compactly supported in $[-\tfrac12,\tfrac12]$, as graphed in Figure \ref{figfg}.

\begin{figure}[H]
\includegraphics[width=3in,height=1.5in]
{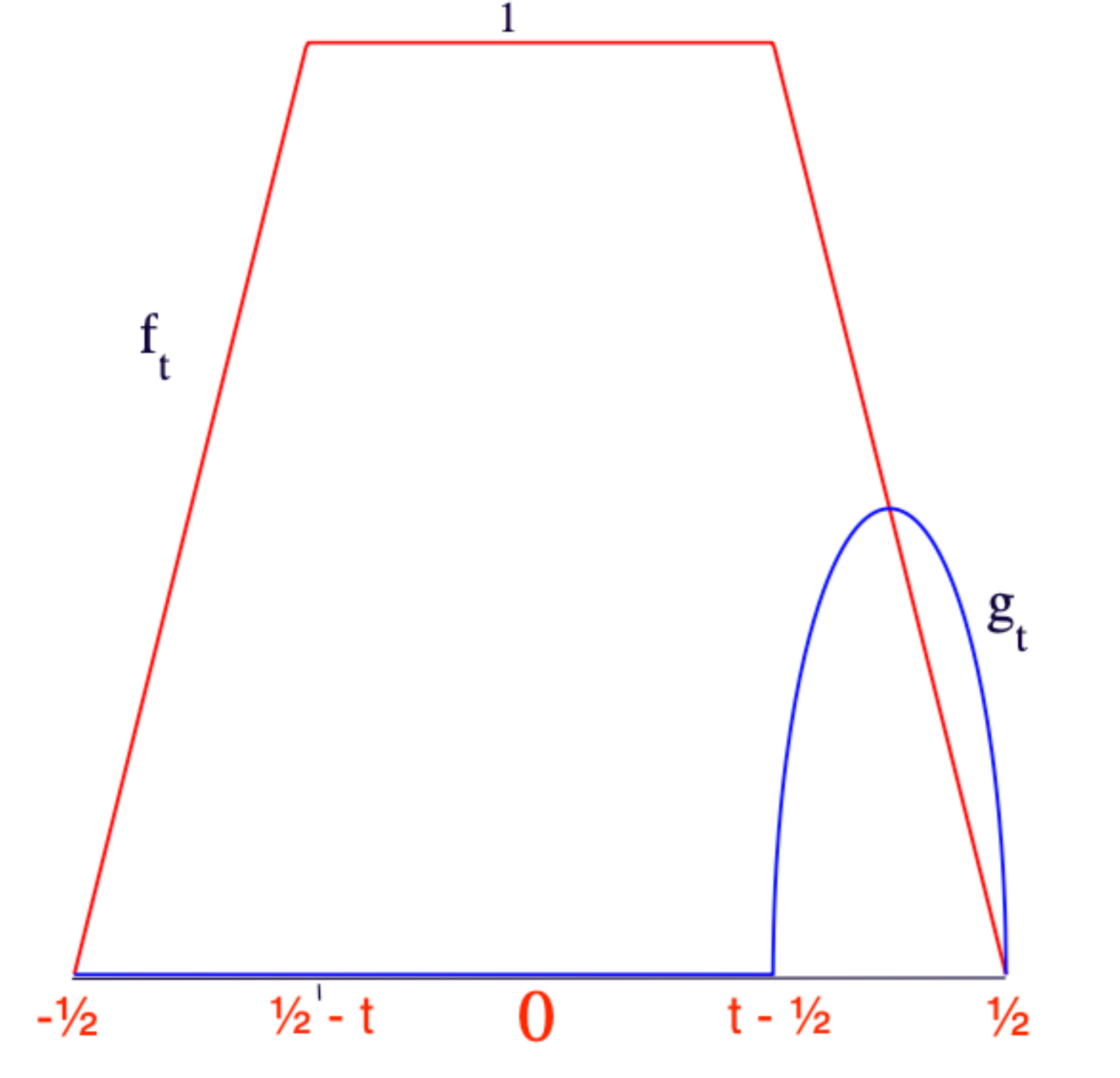} 
\caption{\SMALL{Graphs of $f_t$ and $g_t$. }}\label{figfg}
\end{figure}
One requires 
\begin{equation}\label{ftxt}
f_t(x+t) = 1-f_t(x)
\end{equation}
for $-\tfrac12 \le x \le \tfrac12 - t,$ sets $f_t = 1$ on $\tfrac12 - t \le x \le t - \tfrac12$, and $f_t(-x)=f_t(x)$ for $-\tfrac12 \le x \le \tfrac12$.\footnote{This is a slight modification of the usual constructions; e.g., see proof of Lemma 2.1 \ccite{SWa}.} The condition \eqref{ftxt} also holds for $-\tfrac12 \le x \le 0$ (trivially for $\tfrac12 - t < x \le 0$). Since $f_t$ is even, we have
\begin{equation}\label{ftx1}
f_t(x) + f_t(t - x) = 1, \quad \text{for }  0 \le x \le \tfrac12.
\end{equation}
Taking $x = - \tfrac{t}2$ in \eqref{ftxt} and using the fact that $f_t$ is even, gives 
\begin{equation}\label{ftt2}
f_t(\tfrac{t}2 ) = \tfrac12.
\end{equation}

Further, $g_t$ is defined by $g_t(x) = \sqrt{f_t(x)(1-f_t(x))}$ for $t - \tfrac12 \le x \le \tfrac12,$ and $g_t(x) = 0$ elsewhere. It is easy to check that $g_t(t-x) = g_t(x)$ holds for all $x$ (since $x\to t-x$ leaves its supporting interval 
$[t - \tfrac12, \tfrac12]$ invariant.

\medskip

Note that the Flip invariance of $\mathcal E(t)$ requires that $F_t(U_t) = F_t(U_t^{-1})$, which translates into saying $f_\tau(x) = f_\tau(-x)$, and $G_t(U_t) = G_t(e(t)U_t^{-1})$, i.e. $G_t(x) = G_t(t - x)$ for all $x$. These are easily checked as $G_t$ has period 1 and $x\to t-x$ leaves the interval $[t-\tfrac12,\tfrac12]$ invariant.

\medskip

To give a specific smooth example of $f_t$ (where $t\ge\tfrac12$ is now fixed), choose any $C^\infty$ function $h(x)$ on the closed interval $[-\tfrac12, -\tfrac{t}2]$  such that $h(-\tfrac12) = 0,$ $h(-\tfrac{t}2) = \tfrac12,$ and all derivatives of $h$ vanish at the endpoints $-\tfrac12, -\tfrac{t}2.$ One can then define $f_t$ based on such $h$ by
\[
f_t(x) = 
\begin{cases}
h(x) & -\tfrac12 \le x \le -\tfrac{t}2 \\
1 - h(-x-t) & -\tfrac{t}2 \le x \le \tfrac12 - t \\
\ \ 1 & \tfrac12 - t \le x \le t - \tfrac12  \\
1 - h(x-t) & t - \tfrac12  \le x \le \tfrac{t}2 \\
h(-x) & \tfrac{t}2 \le x \le \tfrac12
\end{cases}
\]
It is easy now to see that $f_t$ is a smooth even function, compactly supported on $[-\tfrac12,\tfrac12],$ and satisfies the condition \eqref{ftx1}.  Of course, the corresponding function $g_t$ will also be $C^\infty$ and compactly supported in the interval indicated by Figure \ref{figfg}. These ensure that the projection $\mathcal E(t)$ is smooth for each $t,$ which we emphasize is defined at $t=\tfrac12$ as well.

\begin{rmk}
It is worthwhile remembering the conditions which ensure that $\mathcal E(t)$ is a projection. Writing $F:=F_t, \ G:=G_t$, those conditions are: both $F,G$ are non-negative, of period 1 (as functions of the real variable $x$), and satisfy
\begin{align*}
F(x) &= F(x)^2 + G(x)^2 + G(x-t)^2 \\ 
G(x)&\big[1 - F(x) - F(x+t)\big] = 0 \\
G(x)&G(x+t) = 0
\end{align*}
for all $x$. These conditions ensure that $\mathcal E(t) = \mathcal E(t)^2 = \mathcal E(t)^*$ is a projection. 
\end{rmk}

The field $\mathcal E(t)$ has trace
\[
\tau(\mathcal E(t)) = \tau(F_t(U_t)) 
= \rmintop_{-\frac12}^{\frac12} f_t(x) dx = t
\]
(in view of Figure \ref{figfg}). In Appendix A (Section 8), the Connes-Chern character of $\mathcal E(t)$ is computed to be
\[
\T(\mathcal E(t)) = (t; \tfrac12, \tfrac12, \tfrac12, \tfrac12)
\]
for $\tfrac12 \le t < 1$.

Now we extend the field $\mathcal E$ to the interval $0 < s < \tfrac12$ by setting
\[
\mathcal F(s) = 1 - \beta_s \mathcal E(1-s)
\]
where $\beta_s: A_{1-s} \to A_s$ is the canonical isomorphism
\[
\beta_s(U_{1-s}) = -U_s, \qquad \beta_s(V_{1-s}) = -V_s^{-1}
\]
which commutes with the canonical Flip so $\mathcal F(s) \in A_s^\Phi$ is also Flip invariant.
We have $\tau(\mathcal F(s)) = s$ and the unbounded topological invariants of $\mathcal F(s)$ can be obtained from the (easy to check) relation
\begin{equation}\label{phibeta}
\phi_{jk}^{1-s} = (-1)^{jk+j+k} \phi_{jk}^{s} \beta_s.
\end{equation}
This gives us the Connes-Chern character of $\mathcal F$
\[
\T(\mathcal F(s)) = (s; \tfrac12, \tfrac12, \tfrac12, \tfrac12)
\]
for $s\in(0,\tfrac12]$ - so it conveniently has the same unbounded traces as $\mathcal E$.

\subsection{The $K_0$ basis}
Recall that the torus $\mathbb T^2$ acts canonically on the rotation algebra by mapping a pair $(a,b) \in \mathbb T^2$ to the automorphism $U \to a U,\ V \to bV$. It is easy to check that the only such toral automorphisms that commute with the Flip are (aside from the identity)
\begin{align*}
\gamma_1(U) &= -U, \qquad \gamma_1(V)=V
\\
\gamma_2(U) &= U, \qquad \ \ \ \gamma_2(V)=-V
\\
\gamma_3(U) &= -U, \qquad \gamma_3(V)=-V
\end{align*}
where $\gamma_3=\gamma_1\gamma_2$.  These give us Flip invariant projection fields $\gamma_1\mathcal E(t), \gamma_2 \mathcal E(t), \gamma_3 \mathcal E(t)$ defined for $t\in[\tfrac12,1)$
as well as fields $\gamma_1\mathcal F(s), \gamma_2 \mathcal F(s), \gamma_3 \mathcal F(s)$ defined for $s\in(0,\tfrac12]$.  It's easy to check the following relations between the unbounded traces and $\gamma_j$
\[
\phi_{jk} \gamma_1 = (-1)^j \phi_{jk}, \qquad 
\phi_{jk} \gamma_2 = (-1)^k \phi_{jk}, \qquad
\phi_{jk} \gamma_3 = (-1)^{j+k} \phi_{jk}.
\]
As a result, we obtain the invariants for the $\mathcal E$-fields
\begin{align*}
\T(\mathcal E(\theta)) & = (\theta; \tfrac12, \tfrac12, \tfrac12, \tfrac12)
\\
\T(\gamma_1\mathcal E(\theta)) &= (\theta; \tfrac12,  \tfrac12, -\tfrac12, -\tfrac12)
\\
\T(\gamma_2 \mathcal E(\theta)) &= (\theta; \tfrac12, - \tfrac12, \tfrac12, - \tfrac12)
\\
\T(\gamma_3 \mathcal E(\theta)) &= (\theta; \tfrac12, - \tfrac12, - \tfrac12, \tfrac12)
\end{align*}
for $\theta \in [\frac12,1)$. Similarly, for the $\mathcal F$-fields one has
\begin{align*}
\T(\mathcal F(\theta)) & = (\theta; \tfrac12, \tfrac12, \tfrac12, \tfrac12)
\\
\T(\gamma_1\mathcal F(\theta)) &= (\theta; \tfrac12,  \tfrac12, -\tfrac12, -\tfrac12)
\\
\T(\gamma_2 \mathcal F(\theta)) &= (\theta; \tfrac12, - \tfrac12, \tfrac12, - \tfrac12)
\\
\T(\gamma_3 \mathcal F(\theta)) &= (\theta; \tfrac12, - \tfrac12, - \tfrac12, \tfrac12) 
\end{align*}
for $\theta \in (0,\frac12]$.

\medskip

We shall let $P_1 = 1$ (the identity projection field), and write
\[
P_3(\theta) = \begin{cases} 
\mathcal F(\theta) &\text{for } 0 < \theta < \tfrac12
\\
 \mathcal E(\theta) &\text{for } \tfrac12 < \theta < 1
\end{cases}, \qquad
P_4(\theta) = \begin{cases} 
\gamma_1 \mathcal F(\theta) &\text{for } 0 < \theta < \tfrac12
\\
\gamma_1 \mathcal E(\theta) &\text{for } \tfrac12 < \theta < 1
\end{cases}
\]
\[
P_5(\theta) = \begin{cases} 
\gamma_2 \mathcal F(\theta) &\text{for } 0 < \theta < \tfrac12
\\
\gamma_2 \mathcal E(\theta) &\text{for } \tfrac12 < \theta < 1
\end{cases}
\qquad
P_6(\theta) = \begin{cases} 
\gamma_3 \mathcal F(\theta) &\text{for } 0 < \theta < \tfrac12
\\
\gamma_3 \mathcal E(\theta) &\text{for } \tfrac12 < \theta < 1.
\end{cases}
\]
We now construct $P_2(\theta)$ as follows. If $0 < \theta < \frac12$, we set
\[
P_2(\theta) = 
\begin{cases} 
\eta \mathcal F(2\theta) &\text{for } 0 < \theta < \tfrac14
\\
 \eta\mathcal E(2\theta) &\text{for } \tfrac14 < \theta < \tfrac12
\end{cases}
\]
where $\eta: A_{2\theta} \to A_\theta$ is the C*-morphism $\eta(U_{2\theta}) = - U_\theta^2, \ \eta(V_{2\theta}) = V_\theta$. From the relations
\begin{equation}\label{phieta}
\phi_{0k}^\theta \eta = \phi_{0k}^{2\theta} - \phi_{1k}^{2\theta}, \qquad
\phi_{1k}^\theta \eta = 0
\end{equation}
one obtains  
\[
\T(P_2(\theta)) = (2\theta; 0, 0, 0, 0).
\]

Now suppose $\frac12 < \theta < 1$. Then $0 < 1-\theta < \frac12$ so that $P_2(1-\theta) \in A_{1-\theta}$ (as in preceding case), which we compose with $\beta_\theta$ to obtain the projection
\[
P_2(\theta) = \beta_\theta P_2(1-\theta) \ \in \ A_\theta,	\qquad (\tfrac12 < \theta < 1)
\]
thus giving
\[
\T(P_2(\theta)) = (2-2\theta; 0, 0, 0, 0)
\]
in view of equations \eqref{phibeta} and the vanishing unbounded traces of $P_2(1-\theta)$.

From the above it is clear that regardless of whether $\theta < \tfrac12$ or $\tfrac12 < \theta$, the projections $P_j(\theta)$ have the same unbounded traces. 

We now claim that the group $K_0(A_\theta^\Phi) = \mathbb Z^6$ has the following projections
\begin{equation}\label{basis}
P_1=1, \quad P_2(\theta), \quad P_3(\theta), \quad P_4(\theta), \quad P_5(\theta), \quad P_6(\theta)
\end{equation}
as basis. As noted above, their Connes-Chern characters are
\begin{align}\label{Kbasis}
\T(P_1) & = (1; 1, 0, 0, 0) \notag
 \\
 \T(P_2) &= 
 \begin{cases} 
(2\theta; 0, 0, 0, 0) &\text{for } 0 < \theta < \tfrac12
\\
(2-2\theta; 0, 0, 0, 0)	 &\text{for } \tfrac12 < \theta < 1
\end{cases}
	\notag
\\
\T(P_3) &= (\theta; \tfrac12, \tfrac12, \tfrac12, \tfrac12)	
\\
\T(P_4) &= (\theta; \tfrac12,  \tfrac12, -\tfrac12, -\tfrac12)		\notag
\\
\T(P_5) &= (\theta; \tfrac12, - \tfrac12, \tfrac12, - \tfrac12)	\notag
\\
\T(P_6) &= (\theta; \tfrac12, - \tfrac12,  -\tfrac12,  \tfrac12).   \notag
\end{align}
It is known from Proposition 3.2 of \ccite{SWa} that a basis for the range of the Connes-Chern character $\T K_0(A_\theta^\Phi)$ consists of the six vectors
\begin{align*}
&(2; 0, 0, 0, 0)  \\
&(1; 1, 0, 0, 0) = \T(1)  \\
&(1; 0, 1, 0, 0) 	\\
&(1; 0, 0, 1, 0)	\\
&(1; 0, 0, 0, 1)	\\
&\begin{cases} 
(\theta; \tfrac12, -\tfrac12, \tfrac12, -\tfrac12) = \T(P_5(\theta)) &\text{for } 0 < \theta < \tfrac12
\\
(\theta; \tfrac12, \tfrac12, -\tfrac12, -\tfrac12) = \T(P_4(\theta)) &\text{for } \tfrac12 < \theta < 1.
\end{cases}
\end{align*}
It can be checked that these six vectors have the same integral span as the vectors in \eqref{Kbasis}. Therefore, the $P_j(\theta)$'s form a basis for $K_0(A_\theta^\Phi)$. 

\medskip

We do not actually use the exact form of the basis projections $P_j(\theta)$ in computing the traces of their cutdowns by the AC projection $e$. Their topological invariants, together with those of $e$, will be sufficient for that purpose. The interesting thing we learn here is the manner by which these invariants interact to giving invariants for cutdowns, as in Theorems \ref{MaintheoremA}, \ref{Kmatrixthm}.

\bigskip

\subsection{Kishimoto's Theorem} 
We now state Kishimoto's Theorem 2.1 \ccite{AK} in a form appropriate for our purposes and specifically for simple AT-algebras of real rank zero (satisfying UCT).

\begin{thm} (Kishimoto \ccite{AK}, Theorem 2.1.)
Let $A$ be a simple AT-C*-algebra $A$ of real rank zero or a separable simple nuclear purely infinite C*-algebra satisfying the Universal Coefficient Theorem. Assume that $K_0(A)$ is finitely generated by classes of projections $g_1, \dots, g_k$ in $A$, and that $K_1(A)$ is finitely generated by classes of unitaries  $u_1, \dots, u_\ell$ in $A$. Then for each $\epsilon > 0$ and each finite subset $F\subset A$, there exists $\updelta > 0$ and a finite subset $G\subset A$ such that for any pair of projections $e_1, e_2$ in $A$ satisfying
\[
\|e_1 x - x e_1\| < \updelta, \qquad \|e_2 x - x e_2\| < \updelta
\]
for $x \in \{g_1, \dots, g_k\} \cup \{u_1, \dots, u_\ell\} \cup G$, and also satisfying
\begin{equation}\label{ege}
[\chi(e_1g_i e_1)] = [\chi(e_2g_i e_2)] \ \ \in \ K_0(A)
\end{equation}
for $i=1,\dots,k$, and
\[
[u_j e_1 + (1- e_1)] = [u_j e_2 + (1-e_2)] \ \ \in \ K_1(A)
\]
for $j=1,\dots,\ell$, there exists a partial isometry $v$ in $A$ such that 
\[
e_1 = v^*v, \qquad e_2 = vv^*, \qquad \|vx - xv\| < \epsilon
\]
for each $x\in F$.
\end{thm}

\medskip

The conditions \eqref{ege} on the generators $g_j$'s imply that $e_1$ and $e_2$ have the same class in $K_0$. In our case, the algebras ($A_\theta$ and $A_\theta^\Phi$) have the cancellation property so that the projections are Murray von-Neumann equivalent and in fact are unitarily equivalent via a unitary in the algebra.
\medskip

In our case, we apply this theorem to the orbifold $A_\theta^\Phi$ which has vanishing $K_1$ (since it is AF) so we are only concerned with the $K_0$ conditions \eqref{ege} in classifying AC projections with respect to AC Murray von-Neumann equivalence.

To this end, we shall use the basis projections \eqref{basis} for $K_0(A_\theta^\Phi),$ and the $K_0$ classes of their cutdown projections which will be convenient to write as
\[
\chi_i := \chi(e P_i(\theta) e)
\]
(for large enough integer parameters in $e,$ of course).  Since the Connes-Chern character $\T$ mentioned in Section 2.1 is injective on $K_0$, it follows that the central equivalence class of $e$ is fully determined by the canonical traces $\tau(\chi_i)$ and the unbounded traces $\phi_{jk}(\chi_i)$ for $i=1,\dots,6$. We find it convenient to organize these numerical invariants for $e$ into a 6-dimensional trace vector 
\[
\vec\tau(e) =  \begin{bmatrix} 
\tau \chi_1 & \tau \chi_2 & \tau \chi_3 & \tau \chi_4 & \tau \chi_5& \tau \chi_6
\end{bmatrix}  
\]
consisting of the canonical traces of the cutdowns, and a topological K-matrix involving the unbounded traces which we lay out as a $4\times6$ matrix
\[
K(e) = 
\begin{bmatrix} 
\phi_{00}(\chi_1) & \phi_{00}(\chi_2) & \phi_{00}(\chi_3) & \phi_{00}(\chi_4) & 
\phi_{00}(\chi_5) & \phi_{00}(\chi_6)
\\
\phi_{01}(\chi_1) & \phi_{01}(\chi_2) & \phi_{01}(\chi_3) & \phi_{01}(\chi_4) & 
\phi_{01}(\chi_5) & \phi_{01}(\chi_6)
\\
\phi_{10}(\chi_1) & \phi_{10}(\chi_2) & \phi_{10}(\chi_3) & \phi_{10}(\chi_4) & 
\phi_{10}(\chi_5) & \phi_{10}(\chi_6)
\\
\phi_{11}(\chi_1) & \phi_{11}(\chi_2) & \phi_{11}(\chi_3) & \phi_{11}(\chi_4) & 
\phi_{11}(\chi_5) & \phi_{11}(\chi_6)
\end{bmatrix}
\]
and we call it the K-matrix of $e$. Its entries all lie in $\tfrac12\mathbb Z$, where the $i$-th column consists of the unbounded traces of $\chi_i$ (arranged from top to bottom in the same order they appear in the character map $\T$). Therefore, in the Flip orbifold, the central equivalence class of $e$ is fully determined by the pair $\vec \tau(e)$ and $K(e)$. 
\medskip

For instance, the K-matrix of the identity is
\begin{equation}\label{Kidentity}
K(1) = \begin{bmatrix}
1 & 0 & \tfrac12 & \tfrac12 & \tfrac12  & \tfrac12
\\
0 & 0 &\tfrac12 & \tfrac12 & -\tfrac12 & -\tfrac12
\\
0  & 0 & \tfrac12 & -\tfrac12 & \tfrac12 & -\tfrac12
\\
0  & 0 & \tfrac12 & -\tfrac12 & -\tfrac12 & \tfrac12
\end{bmatrix}
\end{equation}
(since $\chi_i = P_i(\theta)$) where the columns are just the (unbounded) topological invariants of the basis projections $P_1, \dots, P_6$. Its canonical trace vector is
\[
\vec\tau(1) = 
\begin{cases}
\begin{bmatrix} 1 & 2\theta & \theta & \theta & \theta & \theta \end{bmatrix} &\text{for } 0 < \theta < \tfrac12, 
\\
\begin{bmatrix} 1 & 2-2\theta & \theta & \theta & \theta & \theta \end{bmatrix} &\text{for }
\tfrac12 < \theta < 1.
\end{cases}
\]
\medskip

If $\alpha$ is a smooth automorphism, the K-invariant of the AC projection $\alpha(e)$ is determined by the $K_0$ classes
\[
[\chi(\alpha(e)P\alpha(e))]  
= \alpha_* [\chi(e \alpha^{-1}(P) e)] 
\]
where $P = P_i$ as in \eqref{basis}. These are determined by their canonical traces
\[
\tau [\chi(\alpha(e)P\alpha(e))] = 
\tau \alpha_* [\chi(e \alpha^{-1}(P) e)] = \tau [\chi(e \alpha^{-1}(P) e)]
\]
(since the canonical trace is unique, $\tau \alpha = \tau$), and the unbounded traces
\begin{equation}\label{automP}
\phi_{jk} [\chi(\alpha(e)P\alpha(e))] = \phi_{jk} \alpha_* [\chi(e \alpha^{-1}(P) e)].
\end{equation}

For the unbounded traces, one needs to calculate $\phi_{jk} \alpha$ in terms of $\phi_{jk}$,  and determine the $K_0$ class of $[\alpha^{-1}(P)]$ in terms of the basis.  This would then allow for the calculation of the K-matrix of $\alpha(e)$ in terms of that of $e$. Here is a useful and relevant case, particularly for Theorem \ref{S3imagesofe}.

\medskip

For canonical automorphisms, such as those arising from SL$(2,\mathbb Z)$ (or even $\gamma_1, \gamma_2, \gamma_3$), $\alpha^{-1}$ permutes the basis $[P_i],$ and $\phi_{jk} \alpha$ is a permutation of the $\phi_{jk}$. (The $\gamma_j$'s would change the signs of some $\phi_{jk}$'s.) Therefore, to obtain the K-matrix of $\alpha(e)$ from the K-matrix of $e,$ one
\medskip

(1) permutes the columns of $K(e)$ according to how $\alpha^{-1}$ acts on $[P_i],$ 

(2) permutes the rows of the result in (1) according to how $\alpha$ acts on $\phi_{jk}$.
\medskip

It can be checked, almost by inspection, that applying this procedure to the identity element, where $\alpha$ is the Fourier and Cubic transforms or $\gamma_j$'s, leaves the K-matrix \eqref{Kidentity} of the identity unchanged (as it should).

\subsection{Two Poisson Lemmas} We shall have need for the following two Poisson lemmas for our later computations.

\medskip

\begin{lem}\label{poisson} (Poisson Summation.) Let $f(x)$ be a continuous function on $\mathbb R$ that is compactly supported. Then for each $x$,
\[
\rmsumop_{n=-\infty}^\infty \ft f(n)\ e^{2\pi inx} \ =\ \rmsumop_{n=-\infty}^\infty f(x+n)
\]
where $\ft f(s) = \rmintop_{\mathbb R} f(t) e(-st) dt$ is the Fourier transform of $f$ over $\mathbb R$.
\end{lem}

\begin{proof}  
The short proof below doesn't require $f$ to have compact support, only that $f$ is integrable and decays at $\pm\infty$ so that $h(x) := \rmsumop_{n=-\infty}^\infty f(x+n)$ is a well-defined integrable periodic function (for instance, for Schwartz functions $f$). The Fourier transform of the 1-periodic function $h$ is
\begin{align*}
\ft h(m) &= \rmintop_{0}^{1}  h(x) e(-mx) dx
= \rmintop_{0}^{1} \rmsumop_{n=-\infty}^\infty f(x+n) e(-mx) dx
= \rmsumop_{n=-\infty}^\infty \rmintop_{0}^{1}  f(x+n) e(-mx) dx
\\
&= \rmsumop_{n=-\infty}^\infty \rmintop_{n}^{n+1}  f(t) e(-mt) dt
=  \rmintop_{-\infty}^{\infty}  f(t) e(-mt) dt = \ft f(m).
\end{align*}
Therefore, by Fourier inversion for $h(x)$ we get
\[
\rmsumop_{n=-\infty}^\infty f(x+n) = h(x) = \rmsumop_{m=-\infty}^\infty \ft h(m)\ e^{2\pi imx}
= \rmsumop_{m=-\infty}^\infty \ft f(m)\ e^{2\pi imx}
\]
as required. 
\end{proof}

\medskip

\begin{lem}\label{poissonparity}
For any Schwartz function $H(x)$ on the real line we have the following forms of the Poisson Summation (for all real $x$),
\[
\rmsumop_{n=-\infty}^\infty \ft H(2n)\ e(nx) \ 
=\ \frac12 \rmsumop_{n=-\infty}^\infty H(\tfrac{x}2+n) + H(\tfrac{x}2+ \tfrac12+n)
\]
\[
\rmsumop_{n=-\infty}^\infty \ft H(2n+1)\ e(nx) \ 
=\ \frac12 e(-\tfrac12x) \rmsumop_{n=-\infty}^\infty H(\tfrac{x}2+n) - H(\tfrac{x}2+ \tfrac12+n)
\]
\end{lem}
\begin{proof} From Poisson Summation for $H$,
\[
\rmsumop_{n=-\infty}^\infty \ft H(n)\ e(nx) \ =\ \rmsumop_{n=-\infty}^\infty H(x+n)
\]
replace $x \to x + \tfrac12$ 
\[
\rmsumop_{n=-\infty}^\infty \ft H(n)\ (-1)^n e(nx)  \ =\ \rmsumop_{n=-\infty}^\infty H(x+ \tfrac12+n)
\]
and add the preceding two equalities to get
\[
\rmsumop_{n=-\infty}^\infty \ft H(n)\ e(nx) [1 + (-1)^n] \ 
=\ \rmsumop_{n=-\infty}^\infty H(x+n) + H(x+ \tfrac12+n)
\]
\[
\rmsumop_{n=-\infty}^\infty \ft H(2n)\ e(2nx) \ 
=\ \frac12 \rmsumop_{n=-\infty}^\infty H(x+n) + H(x+ \tfrac12+n)
\]
replace $x$ by $\frac{x}2$,
\[
\rmsumop_{n=-\infty}^\infty \ft H(2n)\ e(nx) \ 
=\ \frac12 \rmsumop_{n=-\infty}^\infty H(\tfrac{x}2+n) + H(\tfrac{x}2+ \tfrac12+n).
\]
Subtracting we get
\[
\rmsumop_{n=-\infty}^\infty \ft H(n)\ e(nx) [ 1 - (-1)^n] \ 
=\ \rmsumop_{n=-\infty}^\infty H(x+n) - H(x+ \tfrac12+n)
\]
which becomes
\[
\rmsumop_{n=-\infty}^\infty \ft H(2n+1)\ e(2nx) \ 
=\ \frac12 e(-x) \rmsumop_{n=-\infty}^\infty H(x+n) - H(x+ \tfrac12+n)
\]
or
\[
\rmsumop_{n=-\infty}^\infty \ft H(2n+1)\ e(nx) \ 
=\ \frac12 e(-\tfrac{x}2) \rmsumop_{n=-\infty}^\infty H(\tfrac{x}2+n) - H(\tfrac{x}2+ \tfrac12+n)
\]
as desired.
\end{proof}

\textcolor{blue}{\Large \section{The AC Powers-Rieffel Projection}}

In this section we construct the Powers-Rieffel projection and compute its unbounded traces.

\medskip

To build approximately central projections from the continuous field $\mathcal E$ of Section 2.2, we consider, for given integers $q',q,p$ and irrational $\theta$, the natural C*-monomorphism $\zeta = \zeta_{q',q,p,\theta}$ defined by  
\begin{equation}\label{zetamorphism}
\zeta: A_\tau \to A_\theta, \qquad 
\zeta(U_\tau) = U_\theta^{q'} = U^{q'}, \qquad
\zeta(V_\tau) = V_\theta^q = V^q
\end{equation}
where we will write $\tau := q'(q\theta-p)$ for brevity (which shan't be confused with the trace map $\tau$!).  One now uses the field $\mathcal E(t)$ to obtain the Powers-Rieffel projection 
\begin{align}\label{rieffelAC}
e &= \zeta_{q',q,p,\theta} \mathcal E(\tau) \\ 
& = \zeta (G_\tau(U_\tau) V_\tau^{-1} + F_\tau(U_\tau) + V_\tau G_\tau(U_\tau))
\notag
\\
&= G_\tau(U^{q'}) V^{-q} + F_\tau(U^{q'}) + V^q G_\tau(U^{q'})
\end{align}
the K-matrix of which will be computed.

It is not hard to see that $e$ is approximately central in the rotation algebra (e.g., it's easy to see that it approximately commutes with $U$ since $\|V^q U - U V^q\| \to 0$ easily follows).
\medskip

Following \ccite{EL}, we will resurrect this projection as a C*-inner product from a Rieffel equivalence bimodule framework (see equation \eqref{einnerproduct}). This can certainly be done for consecutive pairs of convergents $\tfrac{p}{q} < \theta < \tfrac{p'}{q'},$ where $p'q - pq' = 1$.

\medskip

\begin{rmk}\label{tracecondition} It is known that there are infinitely many pairs of consecutive rational convergents $\tfrac{p}{q} < \theta < \tfrac{p'}{q'}$ such that $q'(q\theta-p)$ is bounded away from 0 and 1. For example, by Lemma 3 of Elliott and Evans \ccite{EE}, for each irrational $\theta$ one can show that $\tfrac15 < q'(q\theta-p) < \tfrac45$ is satisfied for infinitely many such convergent pairs. Therefore, there are infinitely many pairs satisfying one of the inequalities
\[
\tfrac15 < q'(q\theta-p) < \tfrac12  \qquad \text{or} \qquad \tfrac12 < q'(q\theta-p) < \tfrac45.
\]
There is no loss of generality in assuming that the irrational $\theta$ conforms to the latter of these conditions, as we have stipulated in inequality \eqref{qalpha} below.  One can reduce the latter case to the former case as follows. Let's suppose that $\tfrac15 < q'(q\theta-p) < \tfrac12$ for infinitely many convergent pairs. One easily converts this to the former case by looking at the corresponding AC Powers-Rieffel projection of trace given by the complementary quantity
\[
\tfrac12 \ < \ 1- q'(q\theta-p) = q(p' - q'\theta) = q( q'[1-\theta] - (q'-p')) \ < \ \tfrac45. 
\]
\end{rmk}

The following lemma shows how the morphism $\zeta$ relates the unbounded traces of $A_\theta$ and $A_\tau$ in order to compute the topological invariants of the projection $e$ give by \eqref{rieffelAC}.

\begin{lem}\label{phizetas}
With $\zeta: A_\tau \to A_\theta$ the morphism in \eqref{zetamorphism}, we have
\[
\phi_{jk}^\theta \zeta
= \updelta_2^{q'} \updelta_2^{j} \Big[ \phi_{0k}^\tau + \phi_{1k}^\tau \Big] 
+ \updelta_2^{q} \updelta_2^{k} \Big[ \phi_{j0}^\tau + (-1)^j\phi_{j1}^\tau\Big] 
+ (-1)^{pjk} \updelta_2^{q'-1} \updelta_2^{q-1} \phi_{jk}^\tau
\]
where $\tau = q'(q\theta-p)$ and $\phi_{jk}^\tau$ are the unbounded $\Phi$-traces for $A_\tau$.
\end{lem}

\begin{proof} We have 
\begin{align*}
\phi_{jk}^\theta \zeta (U_{\tau}^m V_{\tau}^n) 
&= \phi_{jk}^\theta (U_\theta^{q'm} V_\theta^{qn})
= e(-\tfrac{\theta}2 q'q mn)\,\updelta_2^{q'm-j} \updelta_2^{qn-k}
\\
&= e(-\tfrac{1}2 q'[q\theta - p + p] mn)\,\updelta_2^{q'm-j} \updelta_2^{qn-k}
\\
&= e(-\tfrac{\tau}2 mn) (-1)^{q'pmn} \,\updelta_2^{q'm-j} \updelta_2^{qn-k}
\\
&= e(-\tfrac{\tau}2 mn) (-1)^{jpn} \,\updelta_2^{q'm-j} \updelta_2^{qn-k}
\end{align*}
(since $q'm$ in $(-1)^{q'pmn}$ can be replaced by $j,$ in view of the delta function). Using the identity $\updelta_2^{am-b} = \updelta_2^{a} \updelta_2^{b} + \updelta_2^{a-1}\updelta_2^{m-b},$ we have
\[
\phi_{jk} \zeta (U_{\tau}^m V_{\tau}^n) = 
e(-\tfrac{\tau}2 mn) (-1)^{jpn} 
\Big( \updelta_2^{q'} \updelta_2^{j} + \updelta_2^{q'-1}\updelta_2^{m-j} \Big)
 \Big( \updelta_2^{q} \updelta_2^{k} + \updelta_2^{q-1}\updelta_2^{n-k} \Big)
\]
\[
= e(-\tfrac{\tau}2 mn) (-1)^{jpn}
\Big(
\updelta_2^{q'} \updelta_2^{j}\updelta_2^{q} \updelta_2^{k} 
+ \updelta_2^{q'} \updelta_2^{j} \updelta_2^{q-1}\updelta_2^{n-k} +
\updelta_2^{q'-1}\updelta_2^{m-j}\updelta_2^{q} \updelta_2^{k} + \updelta_2^{q'-1}\updelta_2^{m-j}\updelta_2^{q-1}\updelta_2^{n-k} 
\Big).
\]
The first term $\updelta_2^{q'} \updelta_2^{q} = 0$ vanishes as $q,q'$ are coprime. Also, $\updelta_2^{q'} \updelta_2^{q-1} = \updelta_2^{q'}$ (since if $q'$ is even, $q$ has to be odd, and if $q'$ is odd both vanish), and similarly $\updelta_2^{q'-1} \updelta_2^{q} = \updelta_2^{q}$. Thus we get
\begin{align*}
\phi_{jk} \zeta (U_{\tau}^m V_{\tau}^n) 
&= e(-\tfrac{\tau}2 mn) (-1)^{jpn}
\Big(
 \updelta_2^{q'} \updelta_2^{j} \updelta_2^{n-k} +
\updelta_2^{q} \updelta_2^{m-j} \updelta_2^{k} + 
\updelta_2^{q'-1} \updelta_2^{q-1} \updelta_2^{m-j}\updelta_2^{n-k} 
\Big)
\\
&= e(-\tfrac{\tau}2 mn) 
\Big[
(-1)^{jpn} \updelta_2^{q'} \updelta_2^{j} \updelta_2^{n-k} +
(-1)^{jpn}\updelta_2^{q} \updelta_2^{m-j} \updelta_2^{k} + 
(-1)^{jpn} \updelta_2^{q'-1} \updelta_2^{q-1} \updelta_2^{m-j}\updelta_2^{n-k} 
\Big]
\\
&= e(-\tfrac{\tau}2 mn) 
\Big[
\updelta_2^{q'} \updelta_2^{j} \updelta_2^{n-k} +
(-1)^{jn}\updelta_2^{q} \updelta_2^{m-j} \updelta_2^{k} + 
(-1)^{pjk} \updelta_2^{q'-1} \updelta_2^{q-1} \updelta_2^{m-j}\updelta_2^{n-k} 
\Big]
\end{align*}
(where the middle sign holds since if $q$ is even, $p$ is odd)
\[
= \updelta_2^{q'} \updelta_2^{j} e(-\tfrac{\tau}2 mn) \updelta_2^{n-k} +
\updelta_2^{q} \updelta_2^{k} e(-\tfrac{\tau}2 mn) (-1)^{jn} \updelta_2^{m-j}  + 
(-1)^{pjk} \updelta_2^{q'-1} \updelta_2^{q-1} e(-\tfrac{\tau}2 mn) \updelta_2^{m-j}\updelta_2^{n-k} 
\]
\[
\ \ = \updelta_2^{q'} \updelta_2^{j} \Big[ \phi_{0k}^\tau + \phi_{1k}^\tau\Big](U_{\tau}^m V_{\tau}^n)
+ \updelta_2^{q} \updelta_2^{k} \Big[ \phi_{j0}^\tau + (-1)^j \phi_{j1}^\tau\Big](U_{\tau}^m V_{\tau}^n)  + 
(-1)^{pjk} \updelta_2^{q'-1} \updelta_2^{q-1} \phi_{jk}^\tau(U_{\tau}^m V_{\tau}^n)
\]
therefore we get
\[
\phi_{jk} \zeta
= \updelta_2^{q'} \updelta_2^{j} \Big[ \phi_{0k}^\tau + \phi_{1k}^\tau \Big] 
+ \updelta_2^{q} \updelta_2^{k} \Big[ \phi_{j0}^\tau + (-1)^j\phi_{j1}^\tau\Big] 
+ (-1)^{pjk} \updelta_2^{q'-1} \updelta_2^{q-1} \phi_{jk}^\tau
\]
as claimed.
\end{proof}

In Appendix A (Section 8) we calculated the unbounded traces of the field $\mathcal E(t)$ to be
\[
\phi_{00}(\mathcal E) = \phi_{01}(\mathcal E) = \phi_{10}(\mathcal E) = \phi_{11}(\mathcal E) =  \tfrac12.
\]
Combined with Lemma \ref{phizetas}, we obtain the unbounded traces of our approximately central projection $e = \zeta \mathcal E(\tau)$ to be 
\[
\phi_{jk}^\theta(e) = \phi_{jk}^\tau \zeta(\mathcal E(\tau)) 
= \updelta_2^{q'} \updelta_2^{j}  
+ \tfrac12 \updelta_2^{q} \updelta_2^{k} \Big[ 1 + (-1)^j  \Big] 
+ \tfrac12 (-1)^{pjk} \updelta_2^{q'-1} \updelta_2^{q-1} 
\]
or
\begin{equation}\label{phie}
\phi_{jk}(e) 
= \updelta_2^{q'} \updelta_2^{j}  
+ \updelta_2^{q} \updelta_2^{k} \updelta_2^{j} 
+ \tfrac12 (-1)^{pjk} \updelta_2^{q'-1} \updelta_2^{q-1}.
\end{equation}
Written out, we have
\begin{align}\label{etraces}
\phi_{00}(e) &= \updelta_2^{q'} + \updelta_2^{q} + \tfrac12  \updelta_2^{q'-1} \updelta_2^{q-1}, & \phi_{01}(e) &= \updelta_2^{q'}  + \tfrac12 \updelta_2^{q'-1} \updelta_2^{q-1},	
\\
\phi_{10}(e) &= \tfrac12 \updelta_2^{q'-1} \updelta_2^{q-1}, 
& 
\phi_{11}(e) &= \tfrac12 (-1)^{p} \updelta_2^{q'-1} \updelta_2^{q-1}.	\notag
\end{align}


\textcolor{blue}{\Large \section{The Projection as Rieffel C*-Inner Product}}

In this section our goal is to express the Powers-Rieffel projection $e$ in \eqref{ezetaE} as a C*-inner product by applying Rieffel's equivalence bimodule theorem \ccite{MRb}. Doing so will help facilitate the interaction that $e$ has with any projection $P$ so that the topological invariants of the cutdown $\chi(ePe)$ can be calculated.

\medskip

First, however, we give a quick summary of Rieffel's bimodule background along with needed notation.

\medskip

Let $G=M\times \widehat M$ where $\widehat M$ is the Pontryagin dual group of characters on the locally compact Abelian group $M,$ and $\frak h$ the Heisenberg cocycle on $G$ given by 
\[
\frak h((m,s),(m',s')) = \inner{m}{s'}{}
\]
for $m,m'\in M$ and $s,s'\in \widehat M$. The Heisenberg projective unitary representation $\pi: G \to \mathcal L(L^2(M))$ is given by phase multiplication and translation
\[
[\pi_{(m,s)}f](n) = \langle n,s\rangle f(n+m)
\]
for $f\in L^2(M)$, where $M$ is equipped with its Haar-Plancheral measure (which is unique up to positive scalar multiples). It is projective (with respect to $\frak h$) in the sense that 
\begin{equation}\label{projective}
\pi_x \pi_y = \frak h(x,y) \pi_{x+y}, \qquad \pi_x^* = \frak h(x,x)\pi_{-x}
\end{equation}
for $x,y\in G$. We let $S(M)$ denote Schwartz space of $M$.
\medskip

If $D$ is a discrete lattice subgroup of $G$ (i.e. cocompact), it has the associated twisted group C*-algebra $C^*(D,\frak h)$ of the bounded operators on $L^2(M)$ generated by the unitaries $\pi_x$ for $x\in D$. It is the universal C*-algebra generated by unitaries $\{\pi_x: x\in D\}$ satisfying the projective commutation relations \eqref{projective}. From the latter relation we have
\begin{equation}\label{pixpiy}
\pi_x \pi_y = \frak h(x,y)  \conj{\frak h(y,x)} \pi_y \pi_x
\end{equation}
for $x,y\in D$.  Doing the same for the complementary lattice
\[
D^\perp = \{y\in G: \frak h(x,y)  \conj{\frak h(y,x)}  = 1, \forall x\in D\}
\]
one obtains the C*-algebra $C^*(D^\perp,\conj{\frak h})$ generated by the unitaries $\pi_y^*$ for $y\in D^\perp$ (which also satisfy the preceding commutation relation with $\conj{\frak h}$ in place of $\frak h$).

\medskip

Rieffel's theorem states that the Schwartz space $S(M)$ can be completed to an equivalence (or imprimitivity) $C^*(D,\frak h)$-$C^*(D^\perp,\conj{\frak h})$ bimodule, making these algebras strongly Morita equivalent. 

\medskip

On the C*-algebras $C^*(D,\frak h)$ and $C^*(D^\perp,\conj{\frak h})$ there are canonical Flip automorphisms defined, respectively, by
\[
\Phi(\pi_x) = \pi_{-x}, \qquad \Phi'(\pi_y) = \pi_{-y}
\]
for $x\in D,\, y\in D^\perp$. These can easily be shown to be multiplicative with respect to the C*-inner products in the sense that 
\begin{equation}\label{flipinners}
\Phi \Dinner{f}{g} = \Dinner{\tilde f}{\tilde g}, \qquad 
\Phi' \Dpinner{f}{g} = \Dpinner{\tilde f}{\tilde g}
\end{equation}
where $\tilde f(t) = f(-t)$ and for $f,g \in S(M)$. It is also easy to see that for the left and right module actions one has
\[
\widetilde{af}=  \Phi(a)\tilde f, \qquad \widetilde{hb}=  \tilde h \Phi'(b)
\]
for $a\in C^*(D,\frak h)$ and $b\in C^*(D^\perp,\conj{\frak h})$. The are easy to check by taking, for the first equation, $a = \pi_x, \ x\in D$ (and likewise for the second).

\medskip

We now apply this construction to the locally compact Abelian group $M=\Bbb R \times \Bbb Z_q \times \Bbb Z_{q'}$ and lattice subgroup
\[
D = \mathbb Z\varepsilon_1 + \mathbb Z\varepsilon_2
\]
of $G=M \times M$ generated by the basis vectors
\begin{equation}
\begin{aligned} 
\varepsilon_1 &= (\tfrac\alpha{q}, p, 0 ; \ 0, 0, 0) \\ 
\varepsilon_2 &= (0, \ 0, 1 ; \ 1, 1, 0) 
\end{aligned}
\end{equation}
where $\alpha = q\theta - p$. From $\frak h(\varepsilon_1,\varepsilon_2) = e(\frac\alpha{q} + \frac{p}q) = e(\theta)$, we have associated unitaries generating the irrational rotation algebra:
\[
V = \pi_{\varepsilon_1},\qquad U = \pi_{\varepsilon_2}, \qquad VU = e(\theta)UV
\]
(in view of \eqref{pixpiy}) so that the twisted group C*-algebra $C^*(D,\frak h) \cong A_\theta$, generated by $\pi_{\varepsilon_1},\pi_{\varepsilon_2}$, is just the irrational rotation algebra. The Flip $\Phi,$ as defined above on the unitaries $\pi_x$ agrees with that originally defined: $\Phi(U) = U^{-1}, \,\Phi(V) = V^{-1}$.
\medskip

Recall that the measure of each element of $\Bbb Z_q$ is $1/\sqrt q$, so that its total measure is $\sqrt q$. Since a  fundamental domain of the lattice $D$ in $G$ is
\[
[0,\tfrac\alpha{q})\times \Bbb Z_q \times \Bbb Z_{q'} \times [0,1) \times \Bbb Z_q \times \Bbb Z_{q'}
\]
we obtain the covolume of $D$ in $G$ as the product of measures of each component
\[
|G/D| = \frac\alpha{q} \cdot \sqrt{q} \cdot \sqrt{q'} \cdot 1 \cdot \sqrt{q} \cdot \sqrt{q'} = q'\alpha = q'(q\theta-p) =: \tau
\]
which will be the trace of the Powers-Rieffel projection $e$.  

A straightforward computation gives the complementary lattice of $D$ as
\[
D^\perp = \mathbb Z\updelta_1 + \mathbb Z\updelta_2 + \mathbb Z\updelta_3 
\]
with basis vectors
\begin{equation}
\begin{aligned} 
\updelta_1 &= (\tfrac1{qq'}, p, 0 ;\ 0, 0, p') \\
\updelta_2 &= (0, \ 0, 0 ;\ \tfrac1\alpha, q', 0) \\
\updelta_3 &= (0, 0, 1 ;\  0, 0, 0) \\
\end{aligned}
\end{equation}
as readily checked. (Note that $\pi_{\updelta_j}^* = \pi_{-\updelta_j}$ since $\pi_x^* = \frak h(x,x)\pi_{-x}$ as in our case $\frak h(\updelta_j, \updelta_j) = 1$.) We have associated unitaries 
\[
V_1 = \pi_{-\updelta_1}, \qquad V_2 = \pi_{-\updelta_2}, \qquad 
V_3 = \pi_{-\updelta_3}
\]
satisfying the commutation relations
\begin{equation}\label{TheVs}
V_1V_2 = e(\theta') V_2V_1, \qquad V_3V_1 = e(\tfrac{p'}{q'})V_1V_3, 
\qquad V_2V_3=V_3V_2, \qquad V_3^{q'} = 1.
\end{equation}
They generate the C*-algebra $C^*(D^\perp,\conj{\frak h})$ isomorphic to a $q'\times q'$ matrix algebra over some irrational rotation algebra. The Flip $\Phi'(\pi_y) = \pi_{-y}$ on this algebra is easily seen to be given by
\[
\Phi'(V_1) = V_1^{-1}, \qquad \Phi'(V_2) = V_2^{-1}, \qquad \Phi'(V_3) = V_3^{-1}.
\]
The parameter $\theta'$ in \eqref{TheVs} is calculated using \eqref{pixpiy} 
\[
e(\theta') = \pi_{-\updelta_1} \pi_{-\updelta_2} \pi_{-\updelta_1}^*\pi_{-\updelta_2}^*
= \frak h(\updelta_1, \updelta_2)  \conj{\frak h(\updelta_2, \updelta_1) }
\]
giving us (modulo the integers)
\[
\theta' := \frac1{qq'\alpha} +\frac{pq'}q = \frac1{qq'\alpha} +\frac{p'q-1}q 
\ \equiv_{\mathbb Z} \ 
\frac1{qq'\alpha} - \frac{1}q
=\frac{1-q'\alpha}{qq'\alpha}
=\frac{q\alpha'}{qq'\alpha}
=\frac{\alpha'}{q'\alpha}
\]
since $p'q - pq' = 1$ and $q'\alpha + q\alpha' = 1$, where
\[
\alpha' = p'-q'\theta, \qquad  \alpha = q\theta - p.
\]

\bigskip

We now consider the function (as in \ccite{EL})  
\[
 f(t,r,s) = c \delta_q^r \delta_{q'}^s \sqrt{f_0(t)}, \qquad c^2 = \frac{\sqrt{qq'}}\alpha
\] 
where $c$ is a normalizing constant, $f_0$ is continuous and supported on the interval 
$[-\frac1{2q'},\frac1{2q'}]$, and $f_1 = 1$ on $[\frac1{2q'} - \alpha, \alpha-\frac1{2q'}]$, and 
\[
f_0(t-\alpha) = 1 - f_0(t) \qquad  \text{for } \ \alpha - \frac1{2q'} \le t \le \frac1{2q'}
\]
as shown in Figure \ref{fig}. 
\vskip-10pt

\begin{figure}[H]
\includegraphics[width=2.5in,height=1.5in]
{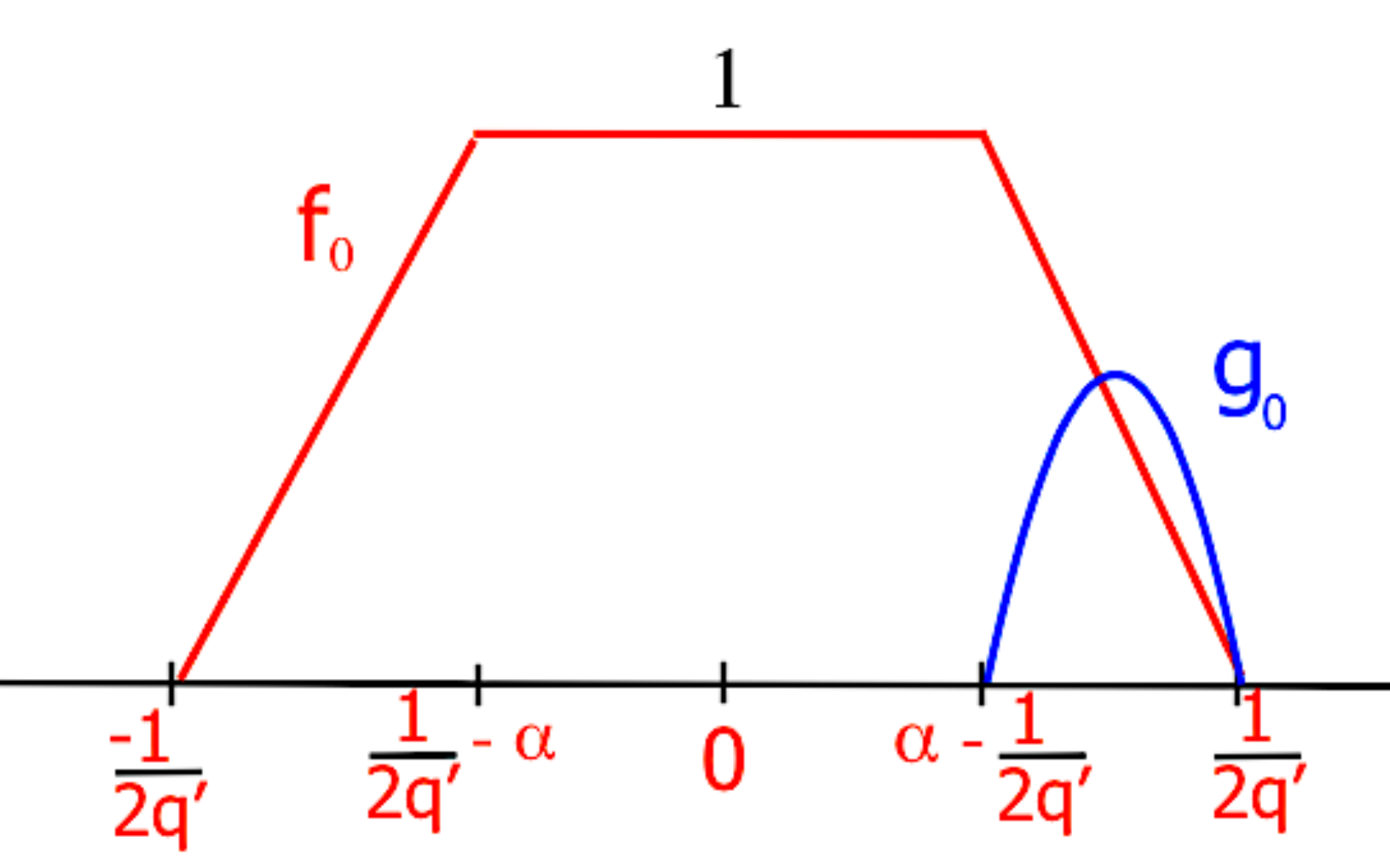} 
\caption{\Small{Graphs of $f_0, g_0$.}}\label{fig}
\end{figure}

According to our Standing Condition 1.3 (in the Introduction), we have
\begin{equation}\label{qalpha}
q\alpha' < \frac12 < q'\alpha = \tau.
\end{equation}
In terms of the function $f_\tau$ defined in Section 2.2, we could in fact take
\begin{equation}\label{fzero}
f_0(t) = f_\tau(q't)
\end{equation}
where $\tau = q'\alpha$.

\bigskip

\subsection{Computation of $\Dpinner{f}{f}$ }

Recall that the $D^\perp$-inner product of $f$ in our setup is
\[
\Dpinner{f}{f} =  
\rmsumop_{n_1, n_2} \ \rmsumop_{n_3=0}^{q'-1} 
\Dpinner{f}{f}(n_1\delta_1 + n_2\delta_2 + n_3\delta_3) 
\pi_{n_1\updelta_1 + n_2\updelta_2 + n_3\updelta_3}^*
\]
where the coefficients will be worked out soon. First, let's find $\pi_{n_1\updelta_1 + n_2\updelta_2 + n_3\updelta_3}^*$.
From $\pi_{u + v} = \conj{\frak h(u, v)} \pi_{u} \pi_{v}$, we get
\begin{align*}
\pi_{n_1\updelta_1 + n_2\updelta_2 + n_3\updelta_3} 
&= \conj{\frak h(n_1\updelta_1, n_2\updelta_2 + n_3\updelta_3)} \pi_{n_1\updelta_1} \pi_{n_2\updelta_2 + n_3\updelta_3}
\\
&= \conj{\frak h(n_1\updelta_1, n_2\updelta_2)} \pi_{\updelta_1}^{n_1} 
\conj{\frak h(n_2\updelta_2, n_3\updelta_3)} \pi_{n_2\updelta_2} \pi_{n_3\updelta_3}
\\
&= e(-n_1 n_2 \theta') \pi_{\updelta_1}^{n_1} \pi_{\updelta_2}^{n_2} \pi_{\updelta_3}^{n_3}
\end{align*}
so
\[
\pi_{n_1\updelta_1 + n_2\updelta_2 + n_3\updelta_3}^*
= e(n_1 n_2 \theta') 
\pi_{\updelta_3}^{-n_3} \pi_{\updelta_2}^{-n_2} \pi_{\updelta_1}^{-n_1} 
= e(n_1 n_2 \theta') V_3^{n_3} V_2^{n_2} V_1^{n_1} 
\]
so the inner product becomes
\[
\Dpinner{f}{f} =  
\rmsumop_{n_1, n_2} \ \rmsumop_{n_3=0}^{q'-1} 
\Dpinner{f}{f}(n_1\delta_1 + n_2\delta_2 + n_3\delta_3) \cdot
e(n_1 n_2 \theta') V_3^{n_3} V_2^{n_2} V_1^{n_1} 
\]
where the coefficients can be worked out as follows:
\begin{align*}
\Dpinner{f}{f}&(n_1\delta_1 + n_2\delta_2 + n_3\delta_3) 
= \Dpinner{f}{f}(\tfrac{n_1}{qq'}, n_1p, n_3 ; \ \ \tfrac{n_2}{\alpha}, n_2 q', n_1 p')
\\
&= \frac1{\sqrt{qq'}} \rmsumop_{r=0}^{q-1} \rmsumop_{s=0}^{q'-1} 
e(\tfrac{n_2q'r}q + \tfrac{n_1p's}{q'}) 
\rmintop_{\Bbb R} \conj{f(t,r,s)} f(t+\tfrac{n_1}{qq'}, r+n_1p, s+n_3) e(t\tfrac{n_2}{\alpha})  dt
\\
&= \frac{c^2}{\sqrt{qq'}} \rmsumop_{r=0}^{q-1} \rmsumop_{s=0}^{q'-1} 
e(\tfrac{n_2q'r}q + \tfrac{n_1p's}{q'}) 
\rmintop_{\mathbb R} \delta_q^r \delta_{q'}^s \sqrt{f_0(t)}
\delta_q^{r+n_1p} \delta_{q'}^{s+n_3} \sqrt{f_0(t+\tfrac{n_1}{qq'})}
  e(t\tfrac{n_2}{\alpha})  dt
\\
&= \frac1{\alpha} \delta_q^{n_1} \delta_{q'}^{n_3}
\rmintop_{\mathbb R} \sqrt{f_0(t) f_0(t+\tfrac{n_1}{qq'})} e(t\tfrac{n_2}{\alpha})  dt.
\end{align*}
From \eqref{fzero} we put $f_0(t) = f_\tau(x)$ where $x = q't$ to get 
\[
\Dpinner{f}{f}(n_1\delta_1 + n_2\delta_2 + n_3\delta_3)  
= \frac1{\tau} \delta_q^{n_1} \delta_{q'}^{n_3}
\rmintop_{\mathbb R} \sqrt{f_\tau(x) f_\tau(x+\tfrac{n_1}{q})} e(x\tfrac{n_2}{\tau})  dx
\]
(as $\tau = q'\alpha$). This gives
\[
\Dpinner{f}{f} =  
\frac1{\tau} \rmsumop_{n_1, n_2} \ \rmsumop_{n_3=0}^{q'-1} 
 \delta_q^{n_1} \delta_{q'}^{n_3}
\rmintop_{\mathbb R} \sqrt{f_\tau(x) f_\tau(x+\tfrac{n_1}{q})} e(x\tfrac{n_2}{\tau})  dx
 \cdot e(n_1 n_2 \theta') V_3^{n_3} V_2^{n_2} V_1^{n_1} 
\]
setting $n_1 = qk$, $n_3 = 0$ (and writing $n_2=m$),
\[
= \frac1{\tau} \rmsumop_{k, m} 
\rmintop_{\mathbb R} \sqrt{f_\tau(x) f_\tau(x+k)} e(x\tfrac{m}{\tau})  dx
 \cdot e(qk m \theta')  V_2^{m} V_1^{qk} 
= \frac1{\tau} \rmsumop_{m} \rmintop_{\mathbb R} f_\tau(x)  e(x\tfrac{m}{\tau})  dx \cdot  V_2^{m}
\]
since the integrand here vanishes for $k\not=0$. The latter integral can be calculated as follows (cf. Figure \ref{figfg} with $t=\tau$)
\begin{align*}
\rmintop_{\mathbb R} f_\tau(x)  e(x\tfrac{m}{\tau})  dx 
&= \rmintop_{-\tfrac12}^{\tfrac12-\tau} f_\tau(x)  e(x\tfrac{m}{\tau})  dx 
+ \rmintop_{\tfrac12-\tau}^{\tau-\tfrac12} f_\tau(x)  e(x\tfrac{m}{\tau})  dx 
+ \rmintop_{\tau-\tfrac12}^{\tfrac12} f_\tau(x)  e(x\tfrac{m}{\tau})  dx 
\\
&= \rmintop_{-\tfrac12}^{\tfrac12-\tau}  f_\tau(x)  e(x\tfrac{m}{\tau})  dx 
+ \rmintop_{\tfrac12-\tau}^{\tau-\tfrac12} e(x\tfrac{m}{\tau})  dx 
+ \rmintop_{-\tfrac12}^{\tfrac12-\tau} f_\tau(x+\tau)  e((x+\tau)\tfrac{m}{\tau})  dx 
\end{align*}
by making the change of variable $x\to x+\tau$ in the third integral. From $f_\tau(x) + f_\tau(x+\tau) = 1$ for $-\tfrac12 \le x \le \tfrac12-\tau,$ we get
\[
\rmintop_{\mathbb R} f_\tau(x)  e(x\tfrac{m}{\tau})  dx
= \rmintop_{-\tfrac12}^{\tfrac12-\tau}e(x\tfrac{m}{\tau})  dx 
+ \rmintop_{\tfrac12-\tau}^{\tau-\tfrac12}  e(x\tfrac{m}{\tau})  dx 
= \rmintop_{-\tfrac12}^{\tau-\tfrac12}e(x\tfrac{m}{\tau})  dx 
= \tau \updelta_{m,0}.
\]
Therefore the $D^\perp$-inner product is $\Dpinner{f}{f} =  \rmsumop_{m}  \updelta_{m,0} \cdot  V_2^{m} = 1.$  This means that the $D$-inner product
\begin{equation}\label{einnerproduct}
 \Dinner{f}{f} = e
\end{equation}
is a projection, which we now proceed to calculate and show to be equal to the Powers-Rieffel projection $e$. 


\subsection{Computation of $\Dinner{f}{f}$ }

Recall that the $D$-inner product of Schwartz functions $f, g$ on $M$ is given by
\[
\inner{f}{g}{D} = |G/D| \rmsumop_{w\in D} \inner{f}{g}{D}(w) \pi_w = 
\tau \rmsumop_{m,n} \inner{f}{g}{D}(m\varepsilon_1+n\varepsilon_2) \  U^n V^m
\]
since $|G/D| = \tau$ and 
\[
\pi_w = \pi_{m\varepsilon_1+n\varepsilon_2} = 
\conj{\frak h(m\varepsilon_1, n\varepsilon_2)} \pi_{m\varepsilon_1} \pi_{n\varepsilon_2} 
= e(-mn\theta)V^m U^n = U^n V^m.
\]
The coefficients are 
\begin{align*}
\Dinner{f}{f}(m\varepsilon_1+n\varepsilon_2) 
&= \Dinner{f}{f}(\tfrac{m\alpha}q, mp, n; \ n, n,0) 
\\ 
&= \rmintop_M f(t,r,s) \conj{f((t,r,s)+(\tfrac{m\alpha}q, mp, n))}\cdot 
\conj{\langle (t,r,s),(n,n,0)\rangle} dt dr ds
\\
&= \rmintop_{\Bbb R \times \Bbb Z_q \times \Bbb Z_{q'}} 
f(t,r,s) \conj{f(t+\tfrac{m\alpha}q, r+mp, s+n)} e(-tn) e(-\tfrac{rn}q) dt \tfrac1{\sqrt q}  \tfrac1{\sqrt {q'}}
\\
&= \frac1{\sqrt{qq'}} \rmsumop_{r=0}^{q-1} \rmsumop_{s=0}^{q'-1} 
e(-\tfrac{rn}q) \rmintop_{\Bbb R} f(t,r,s) \conj{f(t+\tfrac{m\alpha}q, r+mp, s+n)} e(-tn) dt
\\
&= \frac{c^2}{\sqrt{qq'}} \rmsumop_{r=0}^{q-1} \rmsumop_{s=0}^{q'-1} 
e(-\tfrac{rn}q) \rmintop_{\mathbb R} \delta_q^r \delta_{q'}^s  
\delta_q^{r+mp} \delta_{q'}^{s+n} \sqrt{f_0(t) f_0(t+\tfrac{m\alpha}q)} e(-tn) dt
\\
&= \frac1{\alpha} \delta_q^{m} \delta_{q'}^{n}
 \rmintop_{\mathbb R}  \sqrt{f_0(t) f_0(t+\tfrac{m\alpha}q)} e(-tn) dt
\end{align*}
which, again using the change of variable $x=q't$ and $f_0(t) = f_\tau(x),$ gives
\[
\Dinner{f}{f}(m\varepsilon_1+n\varepsilon_2) 
 = \frac1{\tau} \delta_q^{m} \delta_{q'}^{n}
 \rmintop_{\mathbb R}  \sqrt{f_\tau(x) f_\tau(x+m\tfrac{\tau}q)} \ e(-\tfrac{nx}{q'}) dx.
\]

The inner product becomes
\begin{align*}
\inner{f}{f}{D} &=  \tau
\rmsumop_{m,n} \inner{f}{f}{D}(m\varepsilon_1+n\varepsilon_2) \  U^n V^m
= \rmsumop_{m,n} \delta_q^{m} \delta_{q'}^{n}
  \rmintop_{\mathbb R}  \sqrt{f_\tau(x) f_\tau(x+m\tfrac{\tau}q)} e(-\tfrac{nx}{q'}) dx
\  U^n V^m
\\
&=  \rmsumop_{k,\ell} 
  \rmintop_{\mathbb R}\sqrt{f_\tau(x) f_\tau(x+k\tau)} e(-\ell x) dx
\  U^{q'\ell} V^{qk}.
\end{align*}
(by setting $m=qk$ and $n=q'\ell$). It is easy to see that from condition $\frac1{2} < \tau$ (as in \eqref{qalpha}) the integrand here vanishes for $|k| \ge 2$, so the sum over $k$ is concentrated at $k=-1,0,1$, hence the inner product can be written  
\begin{align*}
\inner{f}{f}{D} &=
 \rmsumop_{\ell}  \rmintop_{\mathbb R} \sqrt{f_\tau(x) f_\tau(x-\tau)} e(-\ell x) dx
\  U^{q'\ell} V^{-q}
+  \rmsumop_{\ell} 
  \rmintop_{\mathbb R} f_\tau(x) e(-\ell x) dx \  U^{q'\ell}
\\
&\ \ \ \ +  \rmsumop_{\ell}  \rmintop_{\mathbb R}  \sqrt{f_\tau(x) f_\tau(x+ \tau)} e(-\ell x) dx
\  U^{q'\ell} V^{q}
\end{align*}
or 
\[
\inner{f}{f}{D} = \widetilde G(U^{q'}) V^{-q} +  \widetilde F(U^{q'}) + V^{q} \widetilde G(U^{q'})
\]
where
\[
\widetilde F(z) =  
\rmsumop_{\ell}  \rmintop_{\mathbb R} f_\tau(x) e(-\ell x) dx \cdot z^\ell
=  \rmsumop_{\ell}  \ft f_\tau(\ell) \cdot z^\ell
\]
and
\begin{equation} \label{Gtilde}
\widetilde G(z) =  \rmsumop_{\ell}  \rmintop_{\mathbb R} \sqrt{f_\tau(x) f_\tau(x - \tau)} e(-\ell x) dx \cdot z^\ell.
\end{equation}

In light of the Poisson Lemma \ref{poisson}, we see that $ \widetilde F(z) = F_\tau(z)$ and $\widetilde G(z) = G_\tau(z)$ are the same periodization functions mentioned at the beginning of Section 2.2. Therefore,
\[
\inner{f}{f}{D} = G_\tau(U^{q'}) V^{-q} +  F_\tau(U^{q'}) + V^{q} G_\tau(U^{q'}) \ = \ 
\zeta \mathcal E(\tau) \ = \ e
\]
hence the Powers-Rieffel projection $e$ in \eqref{rieffelAC} is a C*-inner product.

\medskip

\subsection{The Morita Isomorphism}

Now that the projection $e = \Dinner{f}{f}$ is an inner product such that $\Dpinner{f}{f}= 1,$ there is the associated (Morita) isomorphism 
\begin{equation}\label{eta}
\eta: eA_\theta e \longrightarrow C^*(D^\perp,\conj{\frak h}), 
\qquad
\eta(x) = \Dpinner{f}{xf}, 
\qquad \eta^{-1}(y) = \Dinner{f y}{f}
\end{equation}
where we note (and easy to check) that $\eta$ is a homomorphism with respect to the {\it opposite} multiplication on $C^*(D^\perp,\conj{\frak h})$. Further, since the canonical normalized traces $\tau, \tau'$ of $A_\theta$ and $C^*(D^\perp,\conj{\frak h})$ (respectively) are related by $\tau \Dinner gh = \tau(e) \tau' \Dpinner hg,$ one has
\begin{equation}\label{tracesinnerproducts}
\tau(x) = \tau(e) \tau'(\eta(x))
\end{equation}
for $x\in eA_\theta e$. 

Further, this Morita isomorphism intertwines the Flip automorphisms
\begin{equation}\label{etaPhis}
\eta \Phi = \Phi' \eta.
\end{equation}
Indeed, from \eqref{flipinners} for any two Schwartz function $h, k$ we have
\[
\eta \Phi \Dinner{h}{k} = \eta \Dinner{\tilde h}{\tilde k}
= \Dpinner{f}{\Dinner{\tilde h}{\tilde k}f}
\]
which, upon applying the Flip $\Phi'$ (and using $\widetilde{ah}=  \Phi(a)\tilde h$, recalling $\tilde f = f$ is even), gives 
\begin{align*}
\Phi' \eta \Phi \Dinner{h}{k}
&= \Phi' \Dpinner{f}{\Dinner{\tilde h}{\tilde k}f}
= \Dpinner{\tilde f}{ [ \Dinner{\tilde h}{\tilde k}f ]^{\sim} }
\\
&= \Dpinner{\tilde f}{ \Phi(\Dinner{\tilde h}{\tilde k}) \cdot \tilde f  }
= \Dpinner{f}{ \Dinner{h}{k} f  }
\\
&= \eta( \Dinner{h}{k}).
\end{align*}


\textcolor{blue}{\Large \section{Cutdown Approximation}}

In this section we obtain the cutdown approximations 
\begin{equation}\label{cutdownapprox}
\eta(eVe) \approx V_1, \qquad \eta(eUe) \approx V_3
\end{equation}
needed in the next section. (Recall, $a \approx b$ means $\|a - b\| \to 0$ as $q,q'\to\infty$.)

\medskip

Since $V_3$ is unitary of order $q'$ and $V_1$ is a unitary with full spectrum, both satisfying $V_3V_1 = e(\tfrac{p'}{q'})V_1V_3$ (as in \eqref{TheVs}), they generate a C*-subalgebra isomorphic to the  circle algebra $M_{q'}\otimes C(\mathbb T) \cong M_{q'}(C(\mathbb T))$ which, in view of \eqref{cutdownapprox}, approximates the corner algebra $eA_\theta e$. This makes $e$ a circle algebra projection.

\bigskip

From $U = \pi_{\varepsilon_2}$, where $\varepsilon_2 = (0, \ 0, 1 ; \ 1, 1, 0)$, and $f(t,r,s) = c \delta_q^r \delta_{q'}^s \sqrt{f_0(t)}$, we get
\[
(Uf)(t,k,\ell) = (\pi_{\varepsilon_2}f)(t,k,\ell) 
= e(t + \tfrac{k}q) f(t,k,\ell+1)
= c e(t + \tfrac{k}q) \delta_q^{k} \delta_{q'}^{\ell+1} \sqrt{f_0(t)}.
\]
We now calculate the $D^\perp$-inner product coefficient 
\begin{align*}
\Dpinner{f}{Uf}&(n_1\delta_1 + n_2\delta_2 + n_3\delta_3)
= \Dpinner{f}{Uf}(\tfrac{n_1}{qq'}, n_1p, n_3 ; \ \ \tfrac{n_2}{\alpha}, n_2 q', n_1 p')
\\
& = \frac1{\sqrt{qq'}} \rmsumop_{r=0}^{q-1} \rmsumop_{s=0}^{q'-1} 
e(\tfrac{n_2q'r}q + \tfrac{n_1p's}{q'}) 
\rmintop_{\Bbb R} \conj{f(t,r,s)} (Uf)(t+\tfrac{n_1}{qq'}, r+n_1p, s+n_3) e(t\tfrac{n_2}{\alpha})  dt
\\
&= \frac{c}{\sqrt{qq'}} \rmsumop_{r=0}^{q-1} \rmsumop_{s=0}^{q'-1} 
e(\tfrac{n_2q'r}q + \tfrac{n_1p's}{q'}) 
\rmintop_{\Bbb R} \delta_q^r \delta_{q'}^s \sqrt{f_0(t)} 
(Uf)(t+\tfrac{n_1}{qq'}, r+n_1p, s+n_3) 
e(t\tfrac{n_2}{\alpha})  dt
\\
&= \frac{c}{\sqrt{qq'}} 
\rmintop_{\Bbb R}  \sqrt{f_0(t)} (Uf)(t+\tfrac{n_1}{qq'}, n_1p, n_3) 
e(t\tfrac{n_2}{\alpha})  dt
\\
&= \frac{c^2}{\sqrt{qq'}} 
\rmintop_{\Bbb R}  \sqrt{f_0(t)}
e(t + \tfrac{n_1}{qq'} + \tfrac{n_1p}q) \delta_q^{n_1p} \delta_{q'}^{n_3+1} \sqrt{f_0(t+\tfrac{n_1}{qq'})}  e(t\tfrac{n_2}{\alpha})  dt
\end{align*}
since $n_1$ must be divisible by $q$ (if this is nonzero) we can remove $\tfrac{n_1p}q$ 
\[
= \frac{1}{\alpha} \delta_q^{n_1} \delta_{q'}^{n_3+1} e(\tfrac{n_1}{qq'}) 
\rmintop_{\Bbb R}  \sqrt{f_0(t) f_0(t+\tfrac{n_1}{qq'})}  e(t[1+\tfrac{n_2}{\alpha}])  dt
\]
which, in view of $f_0(t) = f_\tau(x)$ where $x=q't$, becomes
\[
= \frac{1}{\tau} \delta_q^{n_1} \delta_{q'}^{n_3+1} e(\tfrac{n_1}{qq'}) 
\rmintop_{\Bbb R}  \sqrt{f_\tau(x) f_\tau(x+\tfrac{n_1}{q})}  e(\tfrac{x}{q'}[1+\tfrac{n_2}{\alpha}])  dx
\]
as $\tau = q'\alpha$ in the first factor. Therefore, the C*-inner product now becomes (as in Section 4.1)
\[
\Dpinner{f}{Uf} =  
\rmsumop_{n_1, n_2} \rmsumop_{n_3=0}^{q'-1} 
\frac{1}{\tau} \delta_q^{n_1} \delta_{q'}^{n_3+1} e(\tfrac{n_1}{qq'}) 
\rmintop_{\Bbb R}  \sqrt{f_\tau(x) f_\tau(x+\tfrac{n_1}{q})}  e(\tfrac{x}{q'}[1+\tfrac{n_2}{\alpha}])  dx
\cdot e(n_1 n_2 \theta') 
V_3^{n_3} V_2^{n_2} V_1^{n_1}  
\]
\[
\ \ \ \ =  
\frac{1}{\tau} V_3^{-1}  \rmsumop_{n_1, n_2}  
\delta_q^{n_1}  e(\tfrac{n_1}{qq'}) 
\rmintop_{\Bbb R}  \sqrt{f_\tau(x) f_\tau(x+\tfrac{n_1}{q})}  e(\tfrac{x}{q'}[1+\tfrac{n_2}{\alpha}])  dx \cdot e(n_1 n_2 \theta') V_2^{n_2} V_1^{n_1}  
\]
now put $n_1=qk$ (and write $n_2=m$)
\[
=  
\frac{1}{\tau} V_3^{-1}  \rmsumop_{k, m}  
  e(\tfrac{k}{q'}) 
\rmintop_{\Bbb R}  \sqrt{f_\tau(x) f_\tau(x+k)}  e(\tfrac{x}{q'}[1+\tfrac{m}{\alpha}])  dx
\cdot e(qk m \theta') V_2^{m} V_1^{qk}.
\]
Since the product $f_\tau(x) f_\tau(x+k)=0$ for $k\not=0$, we get
\[
\Dpinner{f}{Uf} = V_3^{-1}  \cdot \frac{1}{\tau} \rmsumop_{m}  
\rmintop_{\Bbb R}  f_\tau(x) e(\tfrac{x}{q'})  e(\tfrac{mx}{\tau})  dx
\cdot V_2^{m}.
\]
We now use the Poisson Lemma \ref{poisson} to evaluate the sum (where $V_2$ is now replaced by the function $e^{2\pi it},$ for $t\in[0,1],$ since it is a unitary with full spectrum). Making the substitution $x =\tau y,$ we get
\[
C(t) := 
\frac{1}{\tau} \rmsumop_{m}  \rmintop_{\Bbb R}  f_\tau(x) e(\tfrac{x}{q'}) e(\tfrac{mx}{\tau}) dx \cdot e^{2\pi imt} 
= \rmsumop_{m}  \rmintop_{\Bbb R}  f_\tau(\tau y)  e(\tfrac{\tau y}{q'}) e(my) dy \cdot e^{2\pi imt}
\]
or, letting $h(y) = f_\tau(\tau y) e(\tfrac{\tau y}{q'}),$ becomes
\[
C(t) =  \rmsumop_{m}  \rmintop_{\Bbb R}  h(y) e(my) dy \cdot e^{2\pi imt}
= \rmsumop_{m}  \ft h(-m)\, e^{2\pi imt}
= \rmsumop_{m}  \ft h(m)\, e^{-2\pi imt}.
\]
We will show that $C(t)$ converges uniformly in $t$ to 1 for large $q,q',$ where we can take $0\le t \le1$. By Lemma \ref{poisson} we have
\[
C(t) = \rmsumop_{n}  h(n - t)
= \rmsumop_{n}  f_\tau(\tau n - \tau t) e(\tfrac{\tau n - \tau t}{q'}).
\]
Since $f_\tau$ is supported on $[-\tfrac12,\tfrac12],$ the only $n$'s that contribute to this sum are those satisfying $\tau |n - t| < \tfrac12$. Since $\tau > \tfrac12,$ we have $\tfrac12 |n-t| \le \tau |n - t| < \tfrac12$ which gives $-1 < n-t < 1$. Adding this inequality to $0 \le t \le 1$ gives $-1 < n < 2$ so that $n = 0,1$. Using the fact that $f_\tau$ is even, we obtain
\[
C(t) \ = \ f_\tau(\tau t) e(\tfrac{ - \tau t}{q'}) + f_\tau(\tau - \tau t) e(\tfrac{\tau  - \tau t}{q'})
\ \approx \ f_\tau(\tau t) + f_\tau(\tau - \tau t) \ = \ 1
\]
uniformly in $t$ for large $q'$. The last equality here can be seen by looking at the cases $0 \le \tau t \le \tfrac12$ and $\tfrac12 < \tau t \le \tau$ separately. The former case follows from our observation in \eqref{ftx1}, and in the latter case we have $ 0 \le \tau - \tau t < \tau - \tfrac12$ where $f_\tau(\tau - \tau t) = 1$ and $f_\tau(\tau t) = 0$.

This gives us the cutdown approximation
\[
\eta(eUe) =  \Dpinner{f}{Uf} \approx V_3^{-1}.
\]

We next calculate $\Dpinner{f}{Vf}$ where $V = \pi_{\varepsilon_1} 
= \pi_{(\frac\alpha{q}, p, 0 ; \ 0, 0, 0)}$. We have
\[
(Vf)(t,r,s) = f(t+\tfrac\alpha{q}, r+p, s) 
= c \delta_q^{r+p} \delta_{q'}^{s} \sqrt{f_0(t+\tfrac\alpha{q})}.
\]
The C*-inner product coefficients are
\begin{align*}
\Dpinner{f}{Vf}&(n_1\delta_1 + n_2\delta_2 + n_3\delta_3)
\\
&= \frac1{\sqrt{qq'}} \rmsumop_{r=0}^{q-1} \rmsumop_{s=0}^{q'-1} 
e(\tfrac{n_2q'r}q + \tfrac{n_1p's}{q'}) 
\rmintop_{\Bbb R} \conj{f(t,r,s)} Vf(t+\tfrac{n_1}{qq'}, r+n_1p, s+n_3) e(t\tfrac{n_2}{\alpha})  dt
\\
&= \frac{c}{\sqrt{qq'}} \rmsumop_{r=0}^{q-1} \rmsumop_{s=0}^{q'-1} 
e(\tfrac{n_2q'r}q + \tfrac{n_1p's}{q'}) 
\rmintop_{\Bbb R} \delta_q^r \delta_{q'}^s \sqrt{f_0(t)} \cdot Vf(t+\tfrac{n_1}{qq'}, r+n_1p, s+n_3) e(\tfrac{n_2 t}{\alpha})  dt
\intertext{in which both $r, s$ have to be 0}
&= \frac{c}{\sqrt{qq'}} \rmintop_{\Bbb R} \sqrt{f_0(t)} Vf(t+\tfrac{n_1}{qq'}, n_1p, n_3) e(\tfrac{n_2 t}{\alpha})  dt
\\
&= \frac1{\alpha} \delta_q^{n_1+1} \delta_{q'}^{n_3} \rmintop_{\Bbb R}
\sqrt{f_0(t) f_0(t+\tfrac{n_1}{qq'}+\tfrac\alpha{q})} \ e(\tfrac{n_2 t}{\alpha})  dt
\\
&= \frac1{\tau} \delta_q^{n_1+1} \delta_{q'}^{n_3} 
\rmintop_{\Bbb R} \sqrt{f_\tau(x) f_\tau(x+\tfrac{n_1}{q}+\tfrac{\tau}{q})} \ e(\tfrac{n_2 x}{\tau})  dx
\end{align*}
(again using the substitution $x=q't$ and $f_0(t) = f_\tau(x)$), therefore
\[
\Dpinner{f}{Vf} =  \frac1{\tau} 
\rmsumop_{n_1, n_2} \rmsumop_{n_3=0}^{q'-1} 
\delta_q^{n_1+1} \delta_{q'}^{n_3} 
\rmintop_{\Bbb R}  \sqrt{f_\tau(x) f_\tau(x+\tfrac{n_1+\tau}{q})} \ e(\tfrac{n_2 x}{\tau})  dx
\cdot e(n_1 n_2 \theta')\ V_3^{n_3} V_2^{n_2} V_1^{n_1}  
\]
\[
=  \frac1{\tau} 
\rmsumop_{n_1, n_2} \delta_q^{n_1+1}  
\rmintop_{\Bbb R}  \sqrt{f_\tau(x) f_\tau(x+\tfrac{n_1+\tau}{q})} \ e(\tfrac{n_2 x}{\tau})  dx 
  \cdot   V_1^{n_1}  V_2^{n_2}
\]
here we put $n_1 = -1 + qn$ (and $n_2=m$)
\[
=  
V_1^{-1} \frac1{\tau}  \rmsumop_{n, m} 
\rmintop_{\Bbb R}  \sqrt{f_\tau(x) f_\tau(x- \alpha' + n )} \ e(\tfrac{mx}{\tau})  dx 
\cdot  V_1^{nq}  V_2^{m}
\]
since $\tfrac{1-\tau}{q} = \alpha'$. It is easy to see that the integrand here vanishes for $n\not=0,1$. Thus, 
\[
\Dpinner{f}{Vf} = 
V_1^{-1} \frac1{\tau}  \rmsumop_{m} 
\rmintop_{\Bbb R}  \sqrt{f_\tau(x) f_\tau(x- \alpha')} \ e(\tfrac{mx}{\tau})  dx 
\cdot  V_2^{m}
\]
\[ \qquad 
+ V_1^{-1} V_1^{q} \frac1{\tau}  \rmsumop_{m} 
\rmintop_{\Bbb R}  \sqrt{f_\tau(x) f_\tau(x- \alpha' + 1 )} \ e(\tfrac{mx}{\tau})  dx 
\cdot  V_2^{m}
\]
making the substitution $y = x/\tau,$
\[
\Dpinner{f}{Vf} = 
V_1^{-1}  \rmsumop_{m} 
\rmintop_{\Bbb R}  \sqrt{f_\tau(\tau y) f_\tau(\tau y - \alpha')} \ e(my)  dy 
\cdot  V_2^{m}
\]
\[\ \  \qquad \qquad \qquad
+ V_1^{-1} V_1^{q}  \rmsumop_{m} 
\rmintop_{\Bbb R}  \sqrt{f_\tau(\tau y) f_\tau(\tau y - \alpha' + 1 )} \ e(my)  dy 
\cdot  V_2^{m}
\]
\[\ \  \qquad \qquad 
= V_1^{-1}  \rmsumop_{m} \ft h_0(-m) \cdot  V_2^{m} + V_1^{-1} V_1^{q}  \rmsumop_{m} \ft h_1(-m)  \cdot  V_2^{m}
\]
or
\begin{equation}\label{innerfVf}
\Dpinner{f}{Vf} 
= V_1^{-1}  \rmsumop_{m} \ft h_0(m) \cdot  V_2^{-m} + V_1^{-1} V_1^{q}  \rmsumop_{m} \ft h_1(m)  \cdot  V_2^{-m}
\end{equation}
where we have written
\[
h_0(y) = \sqrt{f_\tau(\tau y) f_\tau(\tau y - \alpha')}, \qquad
h_1(y) = \sqrt{f_\tau(\tau y) f_\tau(\tau y - \alpha' + 1)}.
\]
By Lemma \ref{poisson}, the first sum is (as done earlier)
\[
\rmsumop_{m} \ft h_0(m) \ e^{-2\pi i m t} = \rmsumop_{m} h_0(m-t) = h_0(-t) + h_0(1 - t)
\]
since $h_0(m-t)=0$ for $m\not=0,1$ and $0\le t \le1$. (Note $f_\tau(\tau(m-t)) = 0$ for $m\not=0,1$ since $\tfrac12|m-t| \le \tau|m-t| < \tfrac12$ gives $-1 < m - t < 1,$ and adding $0\le t \le1$ gives $-1 < m  < 2$ hence $m=0,1$.)  Since $\alpha' \to 0$ for large $q,q',$ it follows that $h_0(-t) \to f_\tau(\tau t)$ and $h_0(1 - t) \to f_\tau(\tau - \tau t)$ (both uniformly) hence their sum
\[
h_0(-t) + h_0(1 - t) \ \to \ f_\tau(\tau t) + f_\tau(\tau - \tau t) = 1 
\]
by equation \eqref{ftx1} for $f_\tau$. This shows the the first term in \eqref{innerfVf} for $\Dpinner{f}{Vf}$ converges to $V_1^{-1}$ in norm. It remains now to check that the second term converges to 0 for large $q,q'$. For the second term we likewise have only the $m=0,1$ terms
\[
\rmsumop_{m} \ft h_1(m) \ e^{-2\pi i m t} = \rmsumop_{m} h_1(m-t) = h_1(-t) + h_1(1 - t).
\]
The fact that this converges uniformly to 0 follows from $f_\tau(s) f_\tau(s + 1 - \alpha') \to 0$ uniformly in $s$. Since $f_\tau$ is supported on $[-\tfrac12,\tfrac12],$ this product is 0 unless $-\tfrac12 \le s \le -\tfrac12 + \alpha',$ and on this interval (which shrinks to $-\tfrac12$ as $\alpha'\to0$) one has $f_\tau(s) \to f_\tau(-\tfrac12) = 0$. Therefore, we have obtained the norm approximation
\[
\eta(eVe) = \Dpinner{f}{Vf}  \approx V_1^{-1} 
\]
for large $q, q'$.

\textcolor{blue}{\Large \section{K-theory of Powers-Rieffel Projections}}

In this section we prove Theorem \ref{MaintheoremA}.  Theorem \ref{Kmatrixthm} is then proved from it and Lemma \ref{phiV1V3} (which is proved in Section 7).  All norm approximations ``$\approx$" here are understood to hold for large enough integer parameters $q,q'$. 

\medskip

We will denote by $\Xi$ the linear *-anti-automorphism of the continuous field of rotation algebras $\{A_t\}$ defined by
\[
\Xi(U_t^{m} V_t^{n} ) = U_t^{n} V_t^{m}
\]
on the canonical unitary basis. It is a vector space linear transformation satisfying
\[
\Xi(ab) = \Xi(b)\Xi(a), \qquad \Xi(a^*) = \Xi(a)^*
\]
for $a,b \in A_t$. It follows that we also have $\Xi(V_t^{r} U_t^{s} ) = V_t^{s} U_t^{r} $. (We simply write $\Xi$ instead of $\Xi_t$ since it will present no confusion.) Further, $\Xi$ commutes with the Flip 
\[
\Xi \Phi = \Phi \Xi
\]
so it leaves the Flip orbifild invariant.

Let $P(\theta)$ denote any continuous field of Flip-invariant smooth projections.  For example, $P$ could be any of the Powers-Rieffel projection fields forming the basis for $K_0(A_\theta^\Phi) = \mathbb Z^6$ given in \eqref{basis}. For convenience we write $P$ as a  continuous section
\[
P(t) = \rmsumop_{m,n} c_{m, n}(t) U_t^mV_t^{n}
\]
of the continuous field of rotation C*-algebras $\{A_t\},$ where $c_{m,n}(t)$ are rapidly decreasing coefficients; and from its Flip-invariance one has $c_{-m,-n} = c_{m,n}$.

For large $q$, the cut down $e P(\theta) e$ is close to the Flip-invariant projection
\[
\chi(e P(\theta) e) \ \approx \ e P(\theta) e.
\]
where $\chi$ is the characteristic function of the interval $[\tfrac12,\infty]$.
\medskip

Let $A_{p'/q'}$ denote the rational rotation algebra generated by the unitaries $U' = U_{p'/q'}$ and $V' = V_{p'/q'}$ satisfying
\[
V' U' = e(\tfrac{p'}{q'}) U' V'.
\]
Let $\pi$ denote the canonical representation 
\[
\pi: A_{p'/q'} \to C^*(V_1,V_3), \qquad  
\qquad \pi(U') = V_1,\qquad \pi(V') = V_3
\]
which exists since $V_3V_1 = e(\tfrac{p'}{q'})V_1V_3$ from \eqref{TheVs}. 

We will use the well-known result of Elliott (\ccite{GE1984}) that all normalized traces on a rational rotation algebra agree on projections. In particular, $\tau' \pi$ and the canonical trace $\tau_{p'/q'}$ of $A_{p'/q'}$ are equal on projections. The canonical trace of $\chi(e P(\theta) e)$ is therefore obtainable from the approximations $\eta(eUe) \approx V_3^{-1}, \ \eta(eVe) \approx V_1^{-1}$, as follows. Since $P(t)$ is Flip-invariant we can write it as $P(t) = \rmsumop_{m,n} c_{m, n}(t) U_t^{-m}V_t^{-n},$ hence for sufficiently large $q,q',$ we have 
\begin{align*}
\eta(\chi(e P(\theta) e)) 
&\approx \eta(eP(\theta) e)
 =  \eta \rmsumop_{m,n} c_{m,n}(\theta) eU_\theta^{-m}V_\theta^{-n}e 
 \approx  \eta \rmsumop_{m,n} c_{m,n}(\theta) (eU_\theta e)^{-m} (eV_\theta e)^{-n} 
\\
&= \rmsumop_{m,n} c_{m,n}(\theta)  \eta(eV_\theta e)^{-n} \eta (eU_\theta e)^{-m}
\qquad \text{(opposite multiplication)}
\\
&\approx  \rmsumop_{m,n} c_{m,n}(\theta) V_1^{n} V_3^{m}  
\\
&= \pi\ \rmsumop_{m,n} c_{m,n}(\theta) {U'}^n {V'}^m  
\\
&\approx \pi \ \rmsumop_{m,n} c_{m,n}(\tfrac{p'}{q'}) {U'}^n {V'}^m
\\
&\approx \pi \Xi \ \ \rmsumop_{m,n} c_{m,n}(\tfrac{p'}{q'}) {U'}^m {V'}^n
\\
&= \pi \Xi P(\tfrac{p'}{q'}).
\end{align*}
This shows that the projections $\chi(e P(\theta) e)$ and  $\eta^{-1}\pi \Xi P(\tfrac{p'}{q'})$, being close, are therefore unitarily equivalent in the Flip orbifold $A_\theta^\Phi,$ and in particular they have the same canonical and unbounded trace invariants. Thus  
\[
\tau'(\eta(\chi(e P(\theta) e))) = \tau'\pi \Xi P(\tfrac{p'}{q'}) = \tau_{p'/q'}(\Xi P(\tfrac{p'}{q'})) 
= \tau_{p'/q'}(P(\tfrac{p'}{q'})).
\]
where the last equality holds since $\tau_{p'/q'} \Xi$ is a normalized trace on the rational rotation algebra $A_{p'/q'}$ so it agrees with $\tau_{p'/q'}$ on the projections.

Hence from \eqref{tracesinnerproducts}, we get
\begin{equation}\label{canontraceP}
\tau(\chi(e P(\theta) e)) = \tau(e) \tau'(\eta(\chi(e P(\theta) e))) = q'(q\theta-p) \tau_{p'/q'}(P(\tfrac{p'}{q'})).
\end{equation}

\medskip

We now compute the unbounded traces $\phi_{jk}$ of the cutdown projection $\chi(e P(\theta)e)$ (which requires more work). We have
\[
\phi_{jk}(\chi(e P(\theta) e)) = \phi_{jk}( \eta^{-1} \pi \Xi P(\tfrac{p'}{q'}))
= (\phi_{jk} \eta^{-1}) \pi \Xi P(\tfrac{p'}{q'}).
\]
Here, it is easily checked that $\phi_{jk} \eta^{-1}$ is a $\Phi'$-trace when restricted to the C*-algebra generated by $V_1,V_3$ since $\phi_{jk}$ are $\Phi$-traces and using the intertwining relation \eqref{etaPhis}. In Section 9 (see equations \eqref{Phiprimetraces}) we showed that the vector space of $\Phi'$-traces on $C^*(V_1,V_3)$ is 2-dimensional with basis the $\Phi'$-traces 
\[
\psi_{1} ( V_3^{n} V_1^m ) = e(\tfrac{p'mn}{2q'})  \updelta_2^{m},
\qquad
\psi_{2} ( V_3^{n} V_1^m ) = e(\tfrac{p'mn}{2q'}) (-1)^{p'n}\ \updelta_2^{m-q'}.
\]
Therefore, on $C^*(V_1,V_3)$ (the range of $\pi$) there are scalars $a_{jk}^-, a_{jk}^+$ such that
\[
\phi_{jk} \eta^{-1} = a_{jk}^- \psi_1 + a_{jk}^+ \psi_2.
\]
We then have
\[
\phi_{jk}(\chi(e P(\theta) e)) 
= a_{jk}^- \psi_1 \pi \Xi P(\tfrac{p'}{q'}) + a_{jk}^+ \psi_2 \pi \Xi P(\tfrac{p'}{q'}).
\]
The maps $\psi_j \pi \Xi$ ($j=1,2$) are readily found on the basis elements as follows
\begin{align*}
\psi_1 \pi \Xi ({U'}^m {V'}^n) &= \psi_1 \pi({U'}^n {V'}^m)  = \psi_1(V_1^n V_3^m)
= e(-\tfrac{p'mn}{q'}) \psi_1(V_3^m V_1^n)
\\
&= e(-\tfrac{2p'mn}{2q'})  e(\tfrac{p'mn}{2q'})  \updelta_2^{n} 
= e(-\tfrac{p'mn}{2q'})  \updelta_2^{n}
\ =  (\phi_{00} + \phi_{10})({U'}^m {V'}^n)
\end{align*}
hence 
\[
\psi_1 \pi \Xi = \phi_{00}^{} + \phi_{10}^{}
\]
on $A_{p'/q'}$, where here $\phi_{00} \equiv \phi_{00}^{p'/q'}, \phi_{10} \equiv \phi_{10}^{p'/q'}$ -- two of the four basic unbounded traces on the rotation algebra, namely
$\phi_{ij}({U'}^m{V'}^n) = e(-\tfrac{p'mn}{2q'})\,\updelta_2^{m-i} \updelta_2^{n-j}$.
Similarly,
\begin{align*}
\psi_2 \pi \Xi ({U'}^m {V'}^n)  &= \psi_2(V_1^n V_3^m)
= e(-\tfrac{p'mn}{q'}) \psi_2(V_3^mV_1^n)
= e(-\tfrac{p'mn}{2q'}) (-1)^{p'm}\ \updelta_2^{n-q'}
\\
&= (\phi_{0,q'} + (-1)^{p'} \phi_{1,q'}) ({U'}^m {V'}^n)
\end{align*}
since $(-1)^{p'm} = \updelta_2^{m} + (-1)^{p'} \updelta_2^{m-1}$ -- where, of course, $\phi_{0,s}$ is $\phi_{00}$ or $\phi_{01}$ depending on parity of $s$. Therefore, 
\[
\psi_2 \pi \Xi =  \phi_{0,q'} + (-1)^{p'} \phi_{1,q'}.
\]
We have therefore obtained
\begin{equation}\label{phicutdowns}
\phi_{jk}(\chi(e P(\theta) e)) = a_{jk}^- C_0(P) + a_{jk}^+ C_1(P)
\end{equation}
where
\begin{equation}\label{CP}
C_0(P) = \phi_{00}(P) + \phi_{10}(P), \qquad
C_1(P) = \phi_{0,q'}(P) + (-1)^{p'} \phi_{1,q'}(P)
\end{equation}
where $P = P(\tfrac{p'}{q'})$ on the right sides. The invariants $\phi_{jk}P(\tfrac{p'}{q'})$  depend only on the field $P(\theta)$ and do not depend specifically on $\theta, p',q'$ - for instance, this can be seen from unbounded trace values of the basis fields listed in \eqref{Kbasis}. Note however, how $C_1$ depends on the parities of $p'$ and $q'$ in \eqref{CP}.
\medskip

It now remains to find the coefficients $a_{jk}^-, a_{jk}^+$ which depend only on the AC projection $e$. First, evaluate the equation
\[
\phi_{jk} \eta^{-1} = a_{jk}^- \psi_1 + a_{jk}^+ \psi_2
\]
at the identity $\eta(e) = 1$ to get
\begin{equation}\label{stare}
\phi_{jk} (e) = a_{jk}^-   + a_{jk}^+  \updelta_2^{q'}.		
\end{equation}
Next, evaluate it at $V_3,$ where $\psi_{1} ( V_3 ) =  1, \ \psi_{2} ( V_3 ) =  (-1)^{p'} \updelta_2^{q'} = -\updelta_2^{q'}$ (since $p',q'$ are coprime), to get
\begin{equation}\label{star3}
\phi_{jk} \eta^{-1}(V_3) = a_{jk}^-  - a_{jk}^+  \updelta_2^{q'}.	
\end{equation}
Now evaluate it at $V_1$ (noting $\psi_{1} (V_1) = 0,\, \psi_{2} (V_1) =  \updelta_2^{q'-1}$)
\begin{equation}\label{star1}
\phi_{jk} \eta^{-1}(V_1) = a_{jk}^+  \updelta_2^{q'-1}.	
\end{equation}

We now consider the two parity cases for $q'$.

{\bf CASE: $q'$ is even.} Then $p'$ is odd and equations \eqref{stare} and \eqref{star3} become
\[
\phi_{jk} (e) = a_{jk}^-   + a_{jk}^+, \qquad
\phi_{jk} \eta^{-1}(V_3) = a_{jk}^-  - a_{jk}^+
\]
which give
\[
a_{jk}^- = \tfrac12 [ \phi_{jk} (e) + \phi_{jk} \eta^{-1}(V_3) ], \qquad
a_{jk}^+ = \tfrac12 [ \phi_{jk} (e) - \phi_{jk} \eta^{-1}(V_3) ]	\qquad (q' \text{ even}).
\]
From \eqref{phie} we have $\phi_{jk} (e) = \delta_2^j$ ($q'$ even, $q$ odd), and from Lemma \ref{phiV1V3}, and its consequent equation \eqref{phiV3remark} in this case, we have
\[
\phi_{jk} \eta^{-1}(V_3) = \phi_{jk} \inner{fV_3}{f}{D} \ = \  (-1)^{pk} \updelta_2^{j-1} . 
\]
We then get
\[
a_{jk}^- = \tfrac12 [ \delta_2^j + (-1)^{pk} \updelta_2^{j-1}  ], \qquad
a_{jk}^+ = \tfrac12 [ \delta_2^j - (-1)^{pk} \updelta_2^{j-1} ]	      \qquad (q' \text{ even})
\]
which simplify to
\begin{equation}\label{aqprimeeven}
a_{jk}^- = \tfrac12 (-1)^{pjk}, \qquad
a_{jk}^+ = \tfrac12  (-1)^{j+pjk}  		\qquad (q' \text{ even})
\end{equation}
This, together with \eqref{CP}, give us the results indicated in Theorem \ref{MaintheoremA} for the case where $q'$ is even.
\medskip

{\bf CASE: $q'$ is odd.} In this case, equation \eqref{star1} gives 
\[
a_{jk}^+ = \phi_{jk} \eta^{-1}(V_1)
\]
and \eqref{stare} gives 
\begin{equation}\label{aqprimeoddminus}
a_{jk}^- = \phi_{jk} (e)  
= \updelta_2^{q} \updelta_2^{j} \updelta_2^{k}  + \tfrac12 (-1)^{pjk}\updelta_2^{q-1}
\end{equation}
from \eqref{phie} (since $q'$ is odd). By Lemma \ref{phiV1V3}, and consequent equation \eqref{phiV1remark}, we have
\begin{equation}\label{aqprimeoddplus}
a_{jk}^+ = \phi_{jk} \eta^{-1}(V_1) = \phi_{jk} \inner{fV_1}{f}{D}  = \tfrac12 (-1)^{p'j}\, [ \updelta_2^{k-1} + (-1)^{pj} \updelta_2^{q-k-1}  ]
\end{equation}
which are the values given in Theorem \ref{MaintheoremA} in the case that $q'$ is odd. This completes the proof of Theorem \ref{MaintheoremA} (the canonical traces having already been obtained above).

\bigskip

We now proceed to prove Theorem \ref{Kmatrixthm} by calculating the K-matrix of the  projection $e$, which we do for three parity cases. As stated in the Introduction, for simplicity we let $\chi_i := \chi(e P_i(\theta) e)$ denote the cutdown projection of the $i$-th basis generator $P_i$ by $e$. The trace vector of $e$ consists of the traces of these cutdowns
\[
\vec \tau(e) =  \begin{bmatrix} 
\tau \chi_1 & \tau \chi_2 & \tau \chi_3 & \tau \chi_4 & \tau \chi_5& \tau \chi_6
\end{bmatrix}.
\]
In view of \eqref{canontraceP}, they are
\[
\tau\chi_i = q'(q\theta-p) \tau_{p'/q'}(P_i(\tfrac{p'}{q'})).
\]
Inserting the traces of the six fields $P_i$ as indicated in \eqref{Kbasis}, we get the trace vector
\[
\vec \tau(e_q) = 
\begin{bmatrix} 
q'(q\theta-p) & \tau \chi_2 & p'(q\theta-p) & p'(q\theta-p) & p'(q\theta-p) & p'(q\theta-p) 
\end{bmatrix}  
\]
where
\[
\tau \chi_2 =  \begin{cases} 
2p'(q\theta-p) &\text{for } 0 < \theta < \tfrac12
\\
2(q' - p') (q\theta-p)  &\text{for } \tfrac12 < \theta < 1
\end{cases}.
\]
This gives the canonical traces side of the $K$-theory of the AC projection $e$ as stated in Theorem \ref{Kmatrixthm}.

\medskip

In view of equation \eqref{phicutdowns}, it is convenient to write the full K-matrix $K(e) = [\phi_{jk}(\chi_i)]_{jk, i}$ of $e$ (relative to the ordered $K_0$ basis \eqref{basis}) as a matrix product
\[
K(e) = 
\begin{bmatrix} 
\phi_{00}(\chi_1) & \phi_{00}(\chi_2) & \phi_{00}(\chi_3) & \phi_{00}(\chi_4) & 
\phi_{00}(\chi_5) & \phi_{00}(\chi_6)
\\
\phi_{01}(\chi_1) & \phi_{01}(\chi_2) & \phi_{01}(\chi_3) & \phi_{01}(\chi_4) & 
\phi_{01}(\chi_5) & \phi_{01}(\chi_6)
\\
\phi_{10}(\chi_1) & \phi_{10}(\chi_2) & \phi_{10}(\chi_3) & \phi_{10}(\chi_4) & 
\phi_{10}(\chi_5) & \phi_{10}(\chi_6)
\\
\phi_{11}(\chi_1) & \phi_{11}(\chi_2) & \phi_{11}(\chi_3) & \phi_{11}(\chi_4) & 
\phi_{11}(\chi_5) & \phi_{11}(\chi_6)
\end{bmatrix}
\ = \ AC
\]
where
\[
A = \begin{bmatrix} 
a_{00}^- & a_{00}^+  \\ 
a_{01}^- & a_{01}^+  \\ 
a_{10}^- & a_{10}^+  \\ 
a_{11}^- & a_{11}^+ 
 \end{bmatrix}, 		\qquad
C = \begin{bmatrix} 
C_0(P_1) & C_0(P_2) & C_0(P_3) & C_0(P_4) & C_0(P_5) & C_0(P_6) 
\\ 
C_1(P_1) & C_1(P_2) & C_1(P_3) & C_1(P_4) & C_1(P_5) & C_1(P_6) 
\end{bmatrix}.
\]
Of course, $A$ depends only the AC projection $e$, and $C$ is a matrix of topological invariants of the ordered basis (given in \eqref{basis}).
 
\medskip

It is more convenient to consider the following three cases separately: 

(i) $q'$ even,

(ii) $q'$ odd and $q$ even, and 

(iii) $q'$ and $q$ both odd.

\medskip

\noindent{\bf Case (i): $q'$ even.} From \eqref{aqprimeeven} we have $a_{jk}^- = \tfrac12 (-1)^{pjk},\ a_{jk}^+ = \tfrac12  (-1)^{j+pjk},$ so
\[
A = \frac12 \begin{bmatrix} 
1 & 1  \\ 
1 & 1  \\ 
1 & -1  \\ 
(-1)^{p} & (-1)^{p+1} 
 \end{bmatrix}
 \]
 Equations \eqref{CP} in the even $q'$ case (so $p'$ is odd) become
\[
C_0(P) = \phi_{00}(P) + \phi_{10}(P), \qquad
C_1(P) = \phi_{00}(P) - \phi_{10}(P)
\]
which, in view of the unbounded traces in \eqref{Kbasis}, give 
\[
C = \begin{bmatrix} 
1 & 0 & 1 & 0 & 1 & 0 
\\ 
1 & 0 & 0 & 1 & 0 & 1
\end{bmatrix}
\]

Therefore we obtain the K-matrix in the even $q'$ case to be
\[
K(e) = 
\frac12 \begin{bmatrix} 
1 & 1  \\ 
1 & 1  \\ 
1 & -1  \\ 
(-1)^{p} & (-1)^{p+1} 
 \end{bmatrix}
\begin{bmatrix} 
1 & 0 & 1 & 0 & 1 & 0 
\\ 
1 & 0 & 0 & 1 & 0 & 1
\end{bmatrix}
\]
or
\[
\qquad \qquad \qquad K(e) = 
\frac12 
\begin{bmatrix} 
2 & 0 & 1 & 1 & 1 & 1 
\\ 
2 & 0 & 1 & 1 & 1 & 1 
\\ 
0 & 0 & 1 & -1 & 1 & -1
\\ 
0 & 0 & (-1)^p & (-1)^{p+1} & (-1)^p & (-1)^{p+1} 
\end{bmatrix}_{q' \text{ even}}	
\]
(where we have subscripted the matrix with the parity case to which it applies).

\medskip

\noindent{\bf Case (ii): $q$ even.} In this case $p$ and $q'$ are odd and equations \eqref{aqprimeoddminus} and \eqref{aqprimeoddplus} become 
\[
a_{jk}^- = \updelta_2^{j} \updelta_2^{k}, \qquad a_{jk}^+ = \updelta_2^j \updelta_2^{k-1} 
\] 
since
\[
a_{jk}^+ = \tfrac12 (-1)^{p'j}\, [ \updelta_2^{k-1} + (-1)^{pj} \updelta_2^{q-k-1}  ]
= \tfrac12 (-1)^{p'j}\, [ 1 + (-1)^{j}   ] \updelta_2^{k-1} 
=  (-1)^{p'j}\, \updelta_2^j \updelta_2^{k-1} 
=  \updelta_2^j \updelta_2^{k-1}.
\]
This gives
\[
A = \begin{bmatrix} 
1 & 0  \\ 
0 & 1  \\ 
0 & 0  \\ 
0 & 0 
 \end{bmatrix}
 \]
and \eqref{CP} becomes
\[
C_0(P) = \phi_{00}(P) + \phi_{10}(P), \qquad
C_1(P) = \phi_{01}(P) + (-1)^{p'} \phi_{11}(P)
\]
which lead to
\[
C = \begin{bmatrix} 
1 & 0 & 1 & 0 & 1 & 0 
\\ 
0 & 0 & \updelta_2^{p'} & \updelta_2^{p'-1} & -\updelta_2^{p'} & -\updelta_2^{p'-1}
\end{bmatrix}
\]
where we made use of $\tfrac12(1+(-1)^{p'}) = \updelta_2^{p'}$ and 
$\tfrac12(1- (-1)^{p'}) = \updelta_2^{p'-1}$. Therefore,
\[
K(e) = 
\begin{bmatrix} 
1 & 0  \\ 
0 & 1  \\ 
0 & 0  \\ 
0 & 0 
 \end{bmatrix}
 \begin{bmatrix} 
1 & 0 & 1 & 0 & 1 & 0 
\\ 
0 & 0 & \updelta_2^{p'} & \updelta_2^{p'-1} & -\updelta_2^{p'} & -\updelta_2^{p'-1}
\end{bmatrix}
\]
or
\[
K(e) = \begin{bmatrix} 
1 & 0 & 1 & 0 & 1 & 0 
\\ 
0 & 0 & \updelta_2^{p'} & \updelta_2^{p'-1} & -\updelta_2^{p'} & -\updelta_2^{p'-1}
\\ 
0 & 0 & 0 & 0 & 0 & 0
\\ 
0 & 0 & 0 & 0 & 0 & 0
\end{bmatrix}_{q \text{ even}}
\]
\medskip

\noindent{\bf Case (iii): $q',q$ both odd.} Here, equations \eqref{aqprimeoddminus} and \eqref{aqprimeoddplus} give
\[
a_{jk}^-   
=  \tfrac12 (-1)^{pjk}, \qquad 
a_{jk}^+ =  \tfrac12 (-1)^{j + pjk}
\]
since
\[
a_{jk}^+ 
= \tfrac12 (-1)^{p'j}\, [ \updelta_2^{k-1} + (-1)^{pj} \updelta_2^{k}  ]
= \tfrac12 (-1)^{p'j} \, (-1)^{pj(k+1)} 
= \tfrac12 (-1)^{j(p'+p + pk)}
= \tfrac12 (-1)^{j(1 + pk)}
\]
where the last equality holds because one of $p,p'$ will be even and the other odd (from $qp' - q'p = 1,$ where $q,q'$ are both odd). Therefore,
\[
A = \frac12
\begin{bmatrix} 
1 & 1  \\ 
1 & 1  \\ 
1 & -1\\ 
(-1)^p & (-1)^{p-1}
 \end{bmatrix}.
 \]
 The matrix $C$ is the same as in the previous case (ii) (where $q'$ is odd), thus
\[
K(e) = 
\frac12
\begin{bmatrix} 
1 & 1  \\ 
1 & 1  \\ 
1 & -1\\ 
(-1)^p & -(-1)^{p}
 \end{bmatrix}
 \begin{bmatrix} 
1 & 0 & 1 & 0 & 1 & 0 
\\ 
0 & 0 & \updelta_2^{p'} & \updelta_2^{p'-1} & -\updelta_2^{p'} & -\updelta_2^{p'-1}
\end{bmatrix}
\]
\[
= \frac12 \begin{bmatrix} 
1 & 0 & 1+\updelta_2^{p'} & \updelta_2^{p'-1} & 1-\updelta_2^{p'} & -\updelta_2^{p'-1} 
\\ 
1 & 0 & 1+\updelta_2^{p'} & \updelta_2^{p'-1} & 1-\updelta_2^{p'} & -\updelta_2^{p'-1} 
\\ 
1 & 0 & \updelta_2^{p'-1} &  -\updelta_2^{p'-1} & 1+\updelta_2^{p'} & \updelta_2^{p'-1}
\\ 
(-1)^p & 0 & (-1)^p\updelta_2^{p'-1} &  -(-1)^{p}\updelta_2^{p'-1} & (-1)^p(1+\updelta_2^{p'}) & (-1)^p\updelta_2^{p'-1}
\end{bmatrix}_{q',q \text{ both odd}}
\] 
\medskip

Since $q',q$ are both odd, one of $p$ or $p'$ will be even which would simplify the matrix a bit.
(E.g., $(-1)^{p}\updelta_2^{p'-1} = \updelta_2^{p'-1}$ since if $p'$ is odd $p$ must be even. Also, $1 - \updelta_2^{p'} = \updelta_2^{p'-1}$.) With this in mind, the preceding K-matrix becomes
\[
K(e) =  \frac12 \begin{bmatrix} 
1 & 0 & 1+\updelta_2^{p'} & \updelta_2^{p'-1} & \updelta_2^{p'-1} & -\updelta_2^{p'-1} 
\\ 
1 & 0 & 1+\updelta_2^{p'} & \updelta_2^{p'-1} & \updelta_2^{p'-1}& -\updelta_2^{p'-1} 
\\ 
1 & 0 & \updelta_2^{p'-1} &  -\updelta_2^{p'-1} & 1+\updelta_2^{p'} & \updelta_2^{p'-1}
\\ 
(-1)^p & 0 & \updelta_2^{p'-1} &  -\updelta_2^{p'-1} & (-1)^p(1+\updelta_2^{p'}) & \updelta_2^{p'-1}
\end{bmatrix}_{q',q \text{ both odd}}
\]
These all give us the matrices in Theorem \ref{Kmatrixthm} and therefore complete its proof.


\textcolor{blue}{\Large \section{Unbounded Traces of C*-Inner Products}}

In this section we calculate the unbounded traces of the inner products $\inner{fV_1}{f}{D}$ and $\inner{fV_3}{f}{D}$ given by the following lemma. These quantities are needed for the calculations in Section 6 in computing the coefficients $a^-, a^+$.

\bigskip

\begin{lem}\label{phiV1V3} For $ij = 00, 01, 10, 11,$ we have
\[
\phi_{ij} \inner{fV_1}{f}{D}  
=  \frac12 \updelta_2^{j-1} \left[ \updelta_2^{i}+ \updelta_2^{q' - i} (-1)^{p'}  \right] 
+ \frac12 \updelta_2^{q-j-1} \left[ \updelta_2^{i}+ \updelta_2^{q' - i} (-1)^{pq' + p'}  \right]
\]
and
\[
\phi_{ij} \inner{fV_3}{f}{D} \ = \ \frac12 (\updelta_2^{i-1}  + \updelta_2^{q'-1-i}) (\updelta_2^{j} + (-1)^{pi} \updelta_2^{q-j}).
\]
\end{lem}

\begin{rmk}
In Section 6 the computation for $\phi_{ij} \inner{fV_1}{f}{D}$ is needed for the case that $q'$ is odd (see equations \eqref{aqprimeoddminus} and \eqref{aqprimeoddplus}), so in this case it simplifies to
\begin{equation}\label{phiV1remark}
\phi_{ij} \inner{fV_1}{f}{D}  = \frac12 (-1)^{p'i} \left[ \updelta_2^{j-1} + (-1)^{pi} \updelta_2^{q-j-1}  \right]	.	\qquad (q' \text{ odd})
\end{equation}
We have also used the computation for $\phi_{ij} \inner{fV_3}{f}{D}$ for when $q'$ is even (so $q,p'$ are odd), which simplifies it to
\begin{equation}\label{phiV3remark}
\phi_{ij} \inner{fV_3}{f}{D} 	
 \ = \ \updelta_2^{i-1} (-1)^{pj} 		\qquad (q' \text{ even})
\end{equation}
\end{rmk}

\medskip

\subsection{Computation of $\phi_{ij} \inner{fV_1}{f}{D}$}

We will first need to compute $\eta^{-1}(V_1) = \inner{fV_1}{f}{D}$ where (as in Section 4) 
\[
V_1 = \pi_{-\updelta_1}, \qquad \updelta_1 = (\tfrac1{qq'}, p, 0 ;\ 0, 0, p').
\]
Recalling that $f(t,r,s) = c \delta_q^r \delta_{q'}^s \sqrt{f_0(t)}$, where $c^2 = \frac{\sqrt{qq'}}\alpha,$ we have
\begin{align*}
fV_1(t,r,s) &= f\pi_{(\tfrac{-1}{qq'}, -p, 0 ;\ 0, 0, -p')} (t,r,s) =  e(\tfrac{-p's}{q'})  f(t - \tfrac{1}{qq'}, r-p, s) 
= c e(\tfrac{-p's}{q'})   \delta_q^{r-p} \delta_{q'}^{s} \sqrt{f_0(t - \tfrac{1}{qq'})}.
\end{align*}
From the delta factor $\delta_{q'}^{s}$ the exponential appearing here can be replaced by 1, thus
\[
fV_1(t,r,s) = c \delta_q^{r-p} \delta_{q'}^{s} \sqrt{f_0(t - \tfrac{1}{qq'})}.
\]
We therefore get the $D$-inner product coefficients
\begin{align*}
\Dinner{fV_1}{f}&(m\varepsilon_1+n\varepsilon_2) 
= \frac1{\sqrt{qq'}} \rmsumop_{r=0}^{q-1} \rmsumop_{s=0}^{q'-1} 
e(-\tfrac{rn}q) \rmintop_{\Bbb R} fV_1(t,r,s) \conj{f(t+\tfrac{m\alpha}q, r+mp, s+n)} e(-tn) dt
\\
&= \frac{c^2}{\sqrt{qq'}} \rmsumop_{r=0}^{q-1} \rmsumop_{s=0}^{q'-1} 
e(-\tfrac{rn}q) \delta_q^{r-p} \delta_{q'}^{s} \delta_q^{r+mp} \delta_{q'}^{s+n}
\rmintop_{\Bbb R}   \sqrt{f_0(t - \tfrac{1}{qq'}) f_0(t+\tfrac{m\alpha}q)} e(-tn) dt
\intertext{which, in view of the first two delta functions, we set $r=p$ and $s=0$ }
&= \frac1{\alpha} e(-\tfrac{pn}q) \delta_q^{m+1} \delta_{q'}^{n}  
\rmintop_{\Bbb R} \sqrt{f_0(t - \tfrac{1}{qq'}) f_0(t+\tfrac{m\alpha}q)} e(-tn) dt.
\intertext{Using the equality $f_0(t) = f_\tau(q't)$, by \eqref{fzero}, and making the change of variable $x=q't$, this becomes (using $q'\alpha = \tau$)}
&= \frac1{\tau } e(-\tfrac{pn}q) \delta_q^{m+1} \delta_{q'}^{n}  
\rmintop_{\Bbb R} \sqrt{f_\tau(x - \tfrac{1}{q}) f_\tau(x+\tfrac{m\tau}q)}  e(-\tfrac{nx}{q'})dx.
\end{align*}
This gives the $D$-inner product (noting that $|G/D|=\tau=q'\alpha$)
\begin{align*}
\inner{fV_1}{f}{D} &=  \tau
\rmsumop_{m,n} \inner{fV_1}{f}{D}(m\varepsilon_1+n\varepsilon_2) \  U^n V^m
\\
&= \rmsumop_{m,n} e(-\tfrac{pn}q) \delta_q^{m+1} \delta_{q'}^{n}  
\rmintop_{\Bbb R} \sqrt{f_\tau(x - \tfrac{1}{q}) f_\tau(x+\tfrac{m\tau}q)}  e(-\tfrac{nx}{q'})dx
\cdot U^n V^m.
\intertext{Now we set $m=qk-1$ and $n=q'\ell,$}
&= \rmsumop_{k,\ell} e(-\tfrac{pq'\ell}q)  
\rmintop_{\Bbb R} \sqrt{f_\tau(x - \tfrac{1}{q}) f_\tau(x+ k\tau - \tfrac{\tau}q)} 
e(-\ell x) dx \cdot  U^{q'\ell} V^{qk-1}.
\end{align*}
Making the translation $x \to x+\tfrac{1}{q}$ in the integral gives
\[
= \rmsumop_{k,\ell} e(-\tfrac{pq'\ell}q)  
\rmintop_{\Bbb R} \sqrt{f_\tau(x) f_\tau(x + k\tau + \tfrac{1-\tau}q)} 
e(-\ell [x+\tfrac{1}{q}]) dx \cdot  U^{q'\ell} V^{qk-1}
\]
\[
= \rmsumop_{k,\ell} e(-\tfrac{pq'\ell}q)  e(- \tfrac{\ell}{q}) 
\rmintop_{\Bbb R} \sqrt{f_\tau(x) f_\tau(x + k\tau + \tfrac{1-\tau}q)} 
e(-\ell x) dx \cdot  U^{q'\ell} V^{qk-1}.
\]
From $p'q-pq'=1$ we have $e(-\tfrac{pq'\ell}q)  e(- \tfrac{\ell}{q}) = e(-\tfrac{(pq'+1)\ell}q) = e(-\tfrac{p'q\ell}q) = e(-p'\ell) = 1$, and from $1 = q'\alpha + q\alpha' = \tau + q\alpha'$ we have $\tfrac{1-\tau}q = \alpha',$ hence
\[
= \rmsumop_{k,\ell} 
\rmintop_{\Bbb R} \sqrt{f_\tau(x) f_\tau(x + k\tau + \alpha')} 
e(-\ell x) dx \cdot  U^{q'\ell} V^{qk-1}.
\]
Since the function $f_\tau$ is supported on the interval $-\tfrac12 < x < \tfrac12,$ we must also have $-\tfrac12 < x + k\tau + \alpha' < \tfrac12$ (otherwise the integrand vanishes). From these inequalities, we have $k\tau < k\tau + \alpha' < \tfrac12 - x < 1$ which implies $k < \frac1\tau < 2$ since $\tau > \tfrac12$ (in view of Standing Condition 1.3). Further, these inequalities also imply $k\tau + \alpha' > -\tfrac12 - x > -1,$ so $ k\tau > -1 - \alpha'$. But as $\alpha' < \tfrac1{2q}$ (since from $\tau > \tfrac12$ and $1 = \tau + q\alpha'$ one has $q\alpha' < \tfrac12$), we have $k\tau  > -1 - \tfrac1{2q} \ge -\tfrac3{2}$ for $q \ge1$ and hence $k > -\tfrac3{2\tau} > -3$ (from $\tau > \tfrac12$). This shows that the preceding sum runs only over $k = -2, -1, 0, 1$:
\[
\inner{fV_1}{f}{D} = \rmsumop_{k=-2}^1 \rmsumop_{\ell} 
\rmintop_{\Bbb R} H_k(x) e(-\ell x) dx \cdot  U^{q'\ell} V^{qk-1}
= \rmsumop_{k=-2}^1 \rmsumop_{\ell} \ft H_k(\ell) \cdot  U^{q'\ell} V^{qk-1}.
\]
where we have written $H_k(x) = \sqrt{f_\tau(x) f_\tau(x + k\tau + \alpha')}$ for simplicity and used its Fourier transform. 
Applying the unbounded trace $\phi_{ij}(U^mV^n)\ =\ e(-\tfrac{\theta}2 mn)\,\updelta_2^{m-i} \updelta_2^{n-j}$ we get
\begin{align*}
\phi_{ij} \inner{fV_1}{f}{D} 
&= \rmsumop_{k=-2}^1 \rmsumop_{\ell} \ft H_k(\ell)\,  \phi_{ij}(U^{q'\ell} V^{qk-1})
= \rmsumop_{k=-2}^1 \updelta_2^{qk-1-j} 
\rmsumop_{\ell} \ft H_k(\ell) \cdot  e(-\tfrac{\theta}2 q'\ell(qk-1))\,\updelta_2^{q'\ell - i} 
\\
&  = \  \updelta_2^{j-1} (\Omega_0 + \Omega_{-2}) + \updelta_2^{q-1-j} (\Omega_{-1} + \Omega_1)
\end{align*}
where
\[
\Omega_k = \rmsumop_{\ell} \ft H_k(\ell) \cdot  e(-\tfrac{\theta}2 q'\ell(qk-1))\,\updelta_2^{q'\ell - i}.
\]
Writing $\Omega_k$ with respect to even and odd $\ell = 2n, 2n+1,$ gives
\begin{align*}
\Omega_k &= \updelta_2^{i}  \rmsumop_{n} \ft H_k(2n) \cdot  e(-q' \theta(qk-1)n)
 + \updelta_2^{q' - i} e(-\tfrac{\theta}2 q'(qk-1)) 
\rmsumop_{n} \ft H_k(2n+1) \cdot  e(-q' \theta(qk-1)n) 
\\
& = \updelta_2^{i}  \rmsumop_{n} \ft H_k(2n) \cdot  e(nx')
+ \updelta_2^{q' - i} e(\tfrac{1}2 x' ) \rmsumop_{n} \ft H_k(2n+1) \cdot  e(nx') 
\intertext{where $x' = -q' \theta(qk-1)$. By Lemma \ref{poissonparity} this becomes}
& = \frac12 \updelta_2^{i} \rmsumop_n  H_k(\tfrac{x'}2+n) + H_k(\tfrac{x'}2+ \tfrac12+n)
 + \ \frac12 \updelta_2^{q' - i} \rmsumop_n H(\tfrac{x'}2+n) - H(\tfrac{x'}2+ \tfrac12+n)
\end{align*}
so that 
\[
\Omega_k = \frac12 (\updelta_2^{i} + \updelta_2^{q' - i}) \ \rmsumop_n  H_k(\tfrac{x'}2+n) 
+  \frac12 (\updelta_2^{i} - \updelta_2^{q' - i}) \rmsumop_n  H_k(\tfrac{x'}2+ \tfrac12+n).
\]
We work out the two sums in $\Omega_k$ by working out the sum
\[
\rmsumop_n  H_k(\tfrac{x'}2+\tfrac{\epsilon}2 +n) 
= \rmsumop_n  \sqrt{f_\tau(\tfrac{x'}2+\tfrac{\epsilon}2+n) 
f_\tau(\tfrac{x'}2 + \tfrac{\epsilon}2 + n + k\tau + \alpha')} 
\]
where $\epsilon = 0,1$. (This Gossamer of a monster is mostly fur!) It is convenient to write $x'$ (using $q \alpha' = 1 - q' \alpha = 1-\tau$) as follows
\[
x' \ = \ -q' \theta (qk-1) \ = \ -kpq' + p' - k\tau -  \alpha' 
\]
which is easily checked.

\bigskip

The function values $f_\tau(\tfrac{x'}2+\tfrac{\epsilon}2+n)$ and 
$f_\tau(\tfrac{x'}2 + \tfrac{\epsilon}2 + n + k\tau + \alpha')$ are nonzero when both arguments lie in $(-\tfrac12,\tfrac12),$ i.e. when the inequalities 
\[
-1 < x' + \epsilon + 2n < 1, \qquad 
-1< x' + \epsilon + 2n + 2k\tau  + 2\alpha' < 1
\]
hold. Using the above form for $x'$, these inequalities become
\[
-1 <  -kpq' + p' - k\tau -  \alpha' + \epsilon + 2n < 1,
\]
\[
-1< -kpq' + p' - k\tau -  \alpha'  + \epsilon + 2n + 2k\tau + 2\alpha' < 1.
\]
Adding these inequalities gives $-2 < 2[ -kpq' + p'  + \epsilon + 2n ] < 2,$ or
\[
-1 <  -kpq' + p'  + \epsilon + 2n < 1.
\]
Since middle number is an integer we get $-kpq' + p'  + \epsilon + 2n = 0$. For such integer $n$ to exist, the integer $-kpq' + p' + \epsilon$ must be even and the following sum involves only the term with $n = \tfrac12(kpq' - p'  - \epsilon),$ thus we have
\begin{align*}
\rmsumop_n  H_k(\tfrac{x'}2+\tfrac{\epsilon}2 +n)
& = \delta_2^{-kpq' + p' + \epsilon} 
H_k(\tfrac{x'}2+\tfrac{\epsilon}2 + \tfrac12 kpq' - \tfrac12p'  - \tfrac12\epsilon)
\\
&= \delta_2^{kpq' + p' + \epsilon} 
H_k(\tfrac{-kpq' + p' - k\tau -  \alpha'}2 + \tfrac12 kpq' - \tfrac12p')
\\
&= \delta_2^{kpq' + p' + \epsilon} f_\tau(\tfrac{k\tau + \alpha'}2)
\end{align*}
using fact that $f_\tau$ is even in the last equality. This yields
\begin{align*}
\Omega_k &= \frac12 (\updelta_2^{i} + \updelta_2^{q' - i}) \ \delta_2^{kpq' + p'} f_\tau(\tfrac{k\tau + \alpha'}2) 
+  \frac12 (\updelta_2^{i} - \updelta_2^{q' - i}) \delta_2^{kpq' + p' + 1} f_\tau(\tfrac{k\tau + \alpha'}2)
\\
& = \frac12 \left[ (\updelta_2^{i} + \updelta_2^{q' - i}) \ \delta_2^{kpq' + p'}  
+  (\updelta_2^{i} - \updelta_2^{q' - i}) \delta_2^{kpq' + p' + 1} 
\right] f_\tau(\tfrac{k\tau + \alpha'}2)
\\
& = \frac12 \left[ \updelta_2^{i}\delta_2^{kpq' + p'} + \updelta_2^{q' - i} \delta_2^{kpq' + p'}  
+  \updelta_2^{i} \delta_2^{kpq' + p' + 1} - \updelta_2^{q' - i} \delta_2^{kpq' + p' + 1} 
\right] f_\tau(\tfrac{k\tau + \alpha'}2)
\\
& = \frac12 \left[ \updelta_2^{i}+ \updelta_2^{q' - i} ( \delta_2^{kpq' + p'}  
 -  \delta_2^{kpq' + p' + 1} ) \right] f_\tau(\tfrac{k\tau + \alpha'}2)
\end{align*}
therefore
\[
\Omega_k = \frac12 \left[ \updelta_2^{i}+ \updelta_2^{q' - i} (-1)^{kpq' + p'}  \right] f_\tau(\tfrac{k\tau + \alpha'}2).
\]
Setting $k=0$ and $k=-2$ gives
\[
\Omega_0 = \frac12 \left[ \updelta_2^{i}+ \updelta_2^{q' - i} (-1)^{p'}  \right] f_\tau(\tfrac{\alpha'}2), \qquad
\Omega_{-2} = \frac12 \left[ \updelta_2^{i}+ \updelta_2^{q' - i} (-1)^{p'}  \right] f_\tau(\tfrac{\alpha'}2 - \tau)
\]
which sum to
\[
\Omega_{0} + \Omega_{-2} = \frac12 \left[ \updelta_2^{i}+ \updelta_2^{q' - i} (-1)^{p'}  \right]
\]
since $f_\tau(\tfrac{\alpha'}2) + f_\tau(\tfrac{\alpha'}2 - \tau) = 1$ (which follows from \eqref{ftx1} by taking $x = \tfrac{\alpha'}2$ and $t = \tau$ there). Similarly, for $k = -1, 1$ we have
\[
\Omega_{-1} = \frac12 \left[ \updelta_2^{i}+ \updelta_2^{q' - i} (-1)^{pq' + p'}  \right] f_\tau(\tfrac{-\tau + \alpha'}2), \qquad
\Omega_1 = \frac12 \left[ \updelta_2^{i}+ \updelta_2^{q' - i} (-1)^{pq' + p'}  \right] f_\tau(\tfrac{\tau + \alpha'}2)
\]
which sum to
\[
\Omega_{1} + \Omega_{-1} = \frac12 \left[ \updelta_2^{i}+ \updelta_2^{q' - i} (-1)^{pq' + p'}  \right] 
\]
since $f_\tau(\tfrac{-\tau + \alpha'}2) + f_\tau(\tfrac{\tau + \alpha'}2) = 1$ (which follows from \eqref{ftxt} by taking $x = \tfrac{-\tau + \alpha'}2 \in (-\tfrac12,0)$ and $t = \tau$). We have therefore obtained
\begin{align*}
\phi_{ij} \inner{fV_1}{f}{D} 
&= \updelta_2^{j-1} (\Omega_0 + \Omega_{-2}) + \updelta_2^{q-j-1} (\Omega_{-1} + \Omega_1) 
\\
&=  \frac12 \updelta_2^{j-1} \left[ \updelta_2^{i}+ \updelta_2^{q' - i} (-1)^{p'}  \right] +
 \frac12 \updelta_2^{q-j-1} \left[ \updelta_2^{i}+ \updelta_2^{q' - i} (-1)^{pq' + p'}  \right]
\end{align*}
which establishes the equation for $\phi_{ij} \inner{fV_1}{f}{D}$ in the statement of Lemma \ref{phiV1V3}.

Now if $q'$ is odd, as in fact is needed for the computation in Section 6, this result simplifies to
\begin{align*}
\phi_{ij} \inner{fV_1}{f}{D}  &= \frac12  \updelta_2^{q-j-1}  \left[ \updelta_2^{i}+ \updelta_2^{i-1} (-1)^{p + p'}  \right]
+ \frac12 \updelta_2^{j-1}  \left[ \updelta_2^{i}+ \updelta_2^{i-1} (-1)^{p'}  \right]
\\
&= \frac12 \left[ \updelta_2^{q-j-1}   (-1)^{pi + p'i}  + \updelta_2^{j-1}  (-1)^{p'i}  \right]
\\
&= \frac12 (-1)^{p'i}  \left[ \updelta_2^{q-j-1}   (-1)^{pi}  + \updelta_2^{j-1} \right]
\qquad	(q'\text{ odd})
\end{align*}
as noted in the remark following Lemma \ref{phiV1V3}.

\subsection{Computation of $\phi_{ij} \inner{fV_3}{f}{D}$} 

Here we compute the unbounded traces $\phi_{ij}$ of $\eta^{-1}(V_3) = \Dinner{fV_3}{f}$. We have, using $V_3 = \pi_{-\updelta_3}$ and $\updelta_3 = (0, 0, 1 ;\  0, 0, 0)$, 
\[
fV_3(t,r,s) = \pi_{-\updelta_3}(f)(t,r,s)
= \pi_{(0, 0, -1 ;\  0, 0, 0)}(f)(t,r,s) 
= f(t,r,s-1) 
\]
so
\begin{align*}
\Dinner{fV_3}{f}(m\varepsilon_1+n\varepsilon_2) 
&= \frac1{\sqrt{qq'}} \rmsumop_{r=0}^{q-1} \rmsumop_{s=0}^{q'-1} 
e(-\tfrac{rn}q) \rmintop_{\Bbb R} f(t,r,s-1) \conj{f(t+\tfrac{m\alpha}q, r+mp, s+n)} e(-tn) dt
\\
&= \frac{c^2}{\sqrt{qq'}} \rmsumop_{r=0}^{q-1} \rmsumop_{s=0}^{q'-1} 
e(-\tfrac{rn}q) \rmintop_{\Bbb R} \delta_q^r \delta_{q'}^{s-1}  
\delta_q^{r+mp} \delta_{q'}^{s+n} \sqrt{f_0(t) f_0(t+\tfrac{m\alpha}q)} e(-tn) dt
\\
&= \frac1{\alpha} \delta_q^{m} \delta_{q'}^{n+1}
 \rmintop_{\Bbb R}  \sqrt{f_0(t) f_0(t+\tfrac{m\alpha}q)} e(-tn) dt
\intertext{in view of \eqref{fzero}, and using a change of variable $x=q't$, this becomes}
&= \frac1{\tau} \delta_q^{m} \delta_{q'}^{n+1}
 \rmintop_{\Bbb R}   \sqrt{f_\tau(x) f_\tau(x+\tfrac{m\tau}q)} \ e(-\tfrac{nx}{q'}) dx.
\end{align*}

This gives inner product (noting that $|G/D|=\tau=q'\alpha$)
\begin{align*}
\inner{fV_3}{f}{D} &=  \tau
\rmsumop_{m,n} \inner{fV_3}{f}{D}(m\varepsilon_1+n\varepsilon_2) \  U^n V^m
\\
&= \rmsumop_{m,n} \delta_q^{m} \delta_{q'}^{n+1}
 \rmintop_{\Bbb R}   \sqrt{f_\tau(x) f_\tau(x+\tfrac{m\tau}q)} e(-\tfrac{nx}{q'}) dx
\  U^n V^m
\intertext{setting $m=qk$ and $n=q'\ell - 1$}
&=  \rmsumop_{k,\ell} 
\rmintop_{\Bbb R}  \sqrt{f_\tau(x) f_\tau(x+k\tau)} e(-\ell x) e(\tfrac{x}{q'})  dx
\  U^{q'\ell - 1} V^{qk}
\intertext{and noting that the integrand vanishes for $|k|\ge2,$}
&=   \rmsumop_{k=-1}^1 \rmsumop_{\ell} 
\rmintop_{\Bbb R}  L_k(x) e(-\ell x) dx
\  U^{q'\ell - 1} V^{qk}
=   \rmsumop_{k=-1}^1 \rmsumop_{\ell} \ft L_k(\ell) \  U^{q'\ell - 1} V^{qk}
\end{align*}
where we put $L_k(x) = e(\tfrac{x}{q'}) \sqrt{f_\tau(x) f_\tau(x+k\tau)}$.  Applying $\phi_{ij}(U^mV^n)\ =\ e(-\tfrac{\theta}2 mn)\,\updelta_2^{m-i} \updelta_2^{n-j}$ one gets
\[
\phi_{ij} \inner{fV_3}{f}{D} 
= \rmsumop_{k=-1}^1  \updelta_2^{qk-j} \rmsumop_{\ell} \ft L_k(\ell) \cdot
e(-\tfrac{\theta}2 (q'\ell - 1)qk)\, \updelta_2^{q'\ell -1-i} 
= \rmsumop_{k=-1}^1  \updelta_2^{qk-j} e(\tfrac{\theta}2 qk) \ \Lambda_k
\]
where
\[
\Lambda_k = \rmsumop_{\ell} \ft L_k(\ell) \cdot e(-\tfrac{\theta}2 qq' k\ell)\, \updelta_2^{q'\ell -1-i}.
\]
Thus,
\[
\phi_{ij} \inner{fV_3}{f}{D} 
= \updelta_2^{j} \ \Lambda_0 
+ \updelta_2^{q-j} \left[ e(\tfrac{\theta}2 q) \ \Lambda_1
+ e(-\tfrac{\theta}2 q) \ \Lambda_{-1} \right].
\]
We have
\[
\Lambda_k = 
\updelta_2^{i -1}  \rmsumop_{n} \ft L_k(2n) \cdot e(-\theta qq' k n)
+ \updelta_2^{q' -1-i} e(-\tfrac{\theta}2 qq' k) 
\rmsumop_{n} \ft L_k(2n+1) \cdot e(-\theta qq' k n) 
\]
which by the Poisson formulas in Lemma \ref{poissonparity} becomes
\[
\Lambda_k  = \frac12 \updelta_2^{i -1} (A_k + B_k) 
+  \frac12  \updelta_2^{q' -1-i} (A_k - B_k)
\]
where
\[
A_k = \rmsumop_n L_k(\tfrac{-\theta qq' k}2+n), \qquad
B_k = \rmsumop_n L_k(\tfrac{-\theta qq' k}2+ \tfrac12+n).
\]
First, consider $k=0$:
\[
A_0 = \rmsumop_n L_0(n) = \rmsumop_n e(\tfrac{n}{q'}) f_\tau(n) = 1
\]
since $f_\tau(n) = 1$ for $n=0$ and 0 otherwise, and
\[
B_0 = \rmsumop_n L_0(\tfrac12+n) 
= \rmsumop_n e(\tfrac{\tfrac12+n}{q'}) f_\tau(\tfrac12+n)  = 0
\]
since $f_\tau(\tfrac12+n) = 0$ for all integers $n$. This gives
\[
\Lambda_0 = \frac12 (\updelta_2^{i -1} + \updelta_2^{q' -1-i}).
\]
To compute $\Lambda_1$ and $\Lambda_{-1},$ it will suffice to find
\[
A_1 = \rmsumop_n L_1(\tfrac{-\theta qq' }2+n), \qquad
B_1 = \rmsumop_n L_1(\tfrac{-\theta qq' }2+ \tfrac12+n).
\]
First, it is easy to see that $L_{-1}(x) = \conj{L_1(-x)}$ (since $f_\tau$ is even), and consequently $A_{-1} = \conj{A_1}$ and $B_{-1} = \conj{B_1},$ hence $\Lambda_{-1} = \conj{ \Lambda_1}$. For example, to see $B_{-1} = \conj{B_1},$ we have 
\[
B_{-1} = \rmsumop_n L_{-1}(\tfrac{\theta qq'}2+ \tfrac12+n)
= \rmsumop_n \conj{ L_1(\tfrac{-\theta qq'}2 - \tfrac12 - n) }
= \rmsumop_n \conj{ L_1(\tfrac{-\theta qq'}2 + \tfrac12 + n) } = \conj{B_1}
\]
using the substitution $n \to -1-n$. We now show that
\[
A_1 = \frac12 \updelta_2^{q'p} e(\tfrac{-\tau }{2q'}), \qquad B_1 = \frac12 \updelta_2^{q'p-1} e(\tfrac{-\tau}{2q'}). 
\]
For the first, write
\[
A_1 = \rmsumop_n L_1(\tfrac{-\tau}2 - \tfrac{q'p}2 + n).
\]
Letting $\epsilon = 0$ when $q'p$ is even, and $\epsilon = 1$ when $q'p$ is odd, the preceding sum becomes, after appropriate translation of the index $n,$
\[
A_1 = \rmsumop_n L_1(\tfrac{-\tau}2 + \tfrac{\epsilon}2 + n)
= \rmsumop_n e\left(\tfrac{\tfrac{-\tau}2 + \tfrac{\epsilon}2 + n}{q'}\right) 
\sqrt{f_\tau(\tfrac{-\tau}2 + \tfrac{\epsilon}2 + n) f_\tau(\tfrac{\tau}2 + \tfrac{\epsilon}2 + n)}.
\]
The function values under the square-root here is nonzero when both of its arguments lie in the open interval $(-\tfrac12, \tfrac12),$ so their sum $\epsilon+2n$ (which is an integer) is in the open interval $(-1, 1),$ so $\epsilon+2n = 0$. This means that if $q'p$ is odd ($\epsilon=1$) then no such $n$ exists and hence $A_1 = 0$. And if  $q'p$ is even, so $\epsilon=0,$ then $n = 0$:
\[
A_1 =  e(\tfrac{-\tau}{2q'}) \sqrt{f_\tau(\tfrac{-\tau}2) f_\tau(\tfrac{\tau}2)} 
= \frac12 e(\tfrac{-\tau}{2q'})
\]
since $f_\tau(\tfrac{\tau}2 ) = \tfrac12$ from \eqref{ftt2}. We may then write 
\[
A_1 =  \frac12 \updelta_2^{q'p} e(\tfrac{-\tau }{2q'})
\]
in either parity case.

Likewise (with $\epsilon$ as before), we have
\begin{align*}
B_1 
&= \rmsumop_n L_1( \tfrac{-\tau}2 - \tfrac{q'p}2 + \tfrac12+n )
= \rmsumop_n L_1( \tfrac{-\tau}2 + \tfrac{\epsilon}2 + \tfrac12+n )
\\
&= \rmsumop_n e\left(\tfrac{ \tfrac{-\tau}2 + \tfrac{\epsilon}2 + \tfrac12+n}{q'}\right) 
\sqrt{f_\tau( \tfrac{-\tau}2 + \tfrac{\epsilon}2 + \tfrac12+n) 
f_\tau( \tfrac{\tau}2 + \tfrac{\epsilon}2 + \tfrac12+n )}.
\end{align*}
By the same argument as before, the function values under the square-root are nonzero when their arguments are in $(-\tfrac12, \tfrac12),$ so their sum $1+\epsilon+2n = 0$. Therefore, $B_1 = 0$ when $q'p$ is even ($\epsilon=0$). And when $q'p$ is odd we must have $n = -1$ and $\epsilon=1$ which gives $B_1 = e(\tfrac{-\tau}{2q'})  \sqrt{f_\tau( \tfrac{-\tau}2) f_\tau( \tfrac{\tau}2) } = \frac12 e(\tfrac{-\tau}{2q'}),$
so that in either parity case we have
\[
B_1 = \frac12 \updelta_2^{q'p-1} e(\tfrac{-\tau}{2q'}).
\]
We thus get
\begin{align*}
\Lambda_1  &= \frac12 \updelta_2^{i -1} (A_1 + B_1) +  \frac12  \updelta_2^{q' -1-i} (A_1 - B_1)
\\
& = \frac14 \updelta_2^{i -1} \left[  \updelta_2^{q'p}  +  \updelta_2^{q'p-1} \right] 
e(\tfrac{-\tau }{2q'})
+  \frac14  \updelta_2^{q' -1-i} \left[  \updelta_2^{q'p} -  \updelta_2^{q'p-1}  \right] e(\tfrac{-\tau }{2q'})
\\
&  = \frac14 \updelta_2^{i -1} e(\tfrac{-\tau }{2q'})
+  \frac14  \updelta_2^{q' -1-i} (-1)^{q'p} e(\tfrac{-\tau }{2q'})
\end{align*}
and since $e(\tfrac{\theta}2 q) e(\tfrac{-\tau }{2q'}) = (-1)^p$ we have
\[
e(\tfrac{\theta}2 q) \Lambda_1
= \frac14 \updelta_2^{i -1} (-1)^p
+  \frac14  \updelta_2^{q' -1-i} (-1)^{q'p} (-1)^p
= \frac14 \updelta_2^{i -1} (-1)^p
+  \frac14  \updelta_2^{q' -1-i} (-1)^{pi}
\]
(where the last term holds in view of the $\updelta_2^{q' -1-i}$ factor)
which is real, and as seen above $\Lambda_{-1} = \conj{\Lambda_1},$ we get
\begin{align*}
\phi_{ij} \inner{fV_3}{f}{D} 
&= \updelta_2^{j} \ \Lambda_0 
+ \updelta_2^{q-j} \left[ e(\tfrac{\theta}2 q) \ \Lambda_1
+ e(-\tfrac{\theta}2 q) \ \Lambda_{-1} \right]
\\
&= \frac12 \updelta_2^{j}  (\updelta_2^{i -1} + \updelta_2^{q' -1-i}) 
+ \frac12 \updelta_2^{q-j} \cdot \left[ \updelta_2^{i -1} (-1)^p
+   \updelta_2^{q' -1-i} (-1)^{pi} \right]
\\
&= \frac12 \updelta_2^{j}  (\updelta_2^{i -1} + \updelta_2^{q' -1-i}) 
+ \frac12 \updelta_2^{q-j} \cdot \left[ \updelta_2^{i -1} 
+   \updelta_2^{q' -1-i} \right] (-1)^{pi} 
\\
&= \frac12 (\updelta_2^{i-1}  + \updelta_2^{q'-1-i}) (\updelta_2^{j} + (-1)^{pi} \updelta_2^{q-j})
\end{align*}
which is the expression in the statement of Lemma \ref{phiV1V3}, the proof of which is now complete. 

\textcolor{blue}{\Large \section{Appendix A: Unbounded Traces of $\mathcal E(t)$}}

In this section we show that the Connes-Chern character of the continuous field $\mathcal E(t)$ is
\[
\T(\mathcal E(t)) = (t; \tfrac12, \tfrac12, \tfrac12, \tfrac12)
\]
for $\tfrac12 \le t < 1$.

Recall that the continuous field $\mathcal E: [\frac12,1) \to \{A_t\}$ of Flip-invariant Powers-Rieffel projections is given by
\begin{equation}
\mathcal E(t) = G_t(U_t) V_t^{-1} + F_t(U_t) + V_t G_t(U_t) 
\end{equation}
where $F_t, G_t$ are smooth functions, as in Section 2.2. Fix $j,k$ and write $\phi := \phi_{jk}^t,$ which is defined on $A_t^\infty$ by
\begin{equation}
\phi_{jk}^t(U_t^mV_t^n)\ =\ e(-\tfrac{t}2 mn)\,\updelta_2^{m-j} \updelta_2^{n-k}.
\end{equation}
By the $\Phi$-trace property of $\phi,$ we have $\phi(G_t(U_t) V_t^{-1}) = \phi( \Phi(V_t^{-1}) G_t(U_t)) = \phi(V_t G_t(U_t))$ so that
\[
\phi(\mathcal E(t)) = \phi(F_t(U_t)) + 2 \phi(G_t(U_t)V_t^{-1} ).
\]
Expressing $F_t$ in terms of its Fourier transform
\[
F_t(U_t) = \rmsumop_{n\in \mathbb Z} \ft F_t(n) U_t^n 
\]
where $\ft F_t(n) = \ft f_t(n)$ (see proof of Lemma \ref{poisson}), we obtain
\[
\phi(F_t(U_t)) = \rmsumop_{n\in \mathbb Z} \ft F_t(n) \phi(U_t^n) 
= \rmsumop_{n\in \mathbb Z} \ft F_t(n) \delta_2^{n-j}  \delta_2^{0-k}
= \delta_2^{k} \rmsumop_{n\in \mathbb Z} \ft F_t(n) \delta_2^{n-j}.  
\]
Expanding this series into its even and odd indices, one has
\[
\rmsumop_{n\in \mathbb Z} \ft F_t(n) \delta_2^{n-j}  
\ =\ \delta_2^{j}\rmsumop_{n\in \mathbb Z} \ft F_t(2n)   +
\delta_2^{j-1} \rmsumop_{n\in \mathbb Z} \ft F_t(2n+1) 
\ =\ \delta_2^{j}\rmsumop_{n\in \mathbb Z} \ft f_t(2n)   +
\delta_2^{j-1} \rmsumop_{n\in \mathbb Z} \ft f_t(2n+1) 
\]
which, in view of the Poisson Lemma \ref{poissonparity}, become
\[
= \delta_2^{j} \left(
\frac12 \rmsumop_{n=-\infty}^\infty f_t(n) + f_t(\tfrac12+n)
\right)
+ \delta_2^{j-1} \left(
\frac12  \rmsumop_{n=-\infty}^\infty f_t(n) - f_t(\tfrac12+n).
\right)
\]
The sums here are (see Figure \ref{figfg})
\[
\rmsumop_n f_t(n) = f_t(0) = 1, \qquad \rmsumop_n f_t(\tfrac12+n) = 0
\]
so that
\[
\phi(F_t(U_t)) = \delta_2^{k} \frac12 (\delta_2^{j} + \delta_2^{j-1}) = \frac12\delta_2^{k}.
\]
We similarly compute $\phi(G_t(U_t)V_t^{-1} )$ using the Fourier series for $G_t$
\[
G_t(U_t) = \rmsumop_{n\in \mathbb Z} \ft G_t(n) U_t^n. 
\]
We have 
\begin{align*}
\phi(G_t(U_t)V_t^{-1}) &= \rmsumop_{n\in \mathbb Z} \ft G_t(n) \phi(U_t^n V_t^{-1})
= \delta_2^{-1-k} \rmsumop_{n\in \mathbb Z} \ft G_t(n) e(\tfrac{t}2n) \delta_2^{n-j} 
\\
&= \delta_2^{k-1} \delta_2^{j} 
\rmsumop_{n\in \mathbb Z} \ft G_t(2n) e(tn) 
+
\delta_2^{k-1} \delta_2^{j-1} e(\tfrac{t}2)\rmsumop_{n\in \mathbb Z} \ft G_t(2n+1) e(tn) 
\end{align*}
which again by Lemma  \ref{poissonparity} (with $H=g_t$ and using $\ft G_t = \ft g_t$) is 
\[
= \frac12 \delta_2^{k-1} \delta_2^{j} \left( \rmsumop_n g_t(\tfrac{t}2+n) + g_t(\tfrac{t}2+ \tfrac12+n)
\right)
+ \frac12 \delta_2^{k-1} \delta_2^{j-1} \left( \rmsumop_n g_t(\tfrac{t}2+n) - g_t(\tfrac{t}2+ \tfrac12+n) \right).
\]
The individual sums here are (note $f_t(\tfrac{t}2) = \tfrac12 = g_t(\tfrac{t}2)$ from Section 2.2) 
\[
\rmsumop_n g_t(\tfrac{t}2+n) = g_t(\tfrac{t}2) = \frac12, \qquad
\rmsumop_n  g_t(\tfrac{t}2+ \tfrac12+n) = 0  
\]
so
\[
\phi(G_t(U_t)V_t^{-1}) = \frac14 \delta_2^{k-1}(\delta_2^{j} + \delta_2^{j-1}) 
= \frac14 \delta_2^{k-1}.
\]
Therefore we obtain the desired unbounded traces
\[
\phi_{jk}^t (\mathcal E(t)) = \frac12\delta_2^{k} + 2 \frac14 \delta_2^{k-1} = \frac12
\]
which gives the Connes-Chern character of $\mathcal E(t)$ as
\[
\T(\mathcal E(t)) = (t; \tfrac12, \tfrac12, \tfrac12, \tfrac12)
\]
for $\tfrac12 \le t < 1$ (the parameter range over which the field $\mathcal E$ is defined).
\medskip

We note that the unbounded traces of the projection depend on the following values of the underlying function: $f_t(0) = 1,\ f_t(\tfrac12) = 0, \ f_t(\tfrac{t}2)=\tfrac12, \ f_t(\tfrac{t}2+\tfrac12) = 0$ (where the last of these holds since $f_t$ is compactly supported on $[-\tfrac12,\tfrac12],$ where it is an even function). So any homotopic deformation of $f_t$ which preserves these boundary conditions (and of course maintaining the equations between $F,G$ that make $\mathcal E$ is a projection) would still give us the same unbounded traces.


\textcolor{blue}{\Large \section{Appendix B: Flip-Traces on $M_{q'}(C(\mathbb T))$}}

In this section we show that the two functionals
\begin{align}\label{Phiprimetraces}
\psi_{1} ( V_3^{n} V_1^m ) &= e(\tfrac{p'mn}{2q'})  \updelta_2^{m} 
\\
\psi_{2} ( V_3^{n} V_1^m ) &= e(\tfrac{p'mn}{2q'}) (-1)^{p'n}\ \updelta_2^{m-q'}	\notag
\end{align}
form a basis for all $\Phi'$-traces on the circle algebra $B = C^*(V_1,V_3) \cong M_{q'}\otimes C(\mathbb T)$ generated by unitaries $V_3, V_1$ satisfying
\[
V_3 V_1 = e(\tfrac{p'}{q'}) V_1 V_3, \qquad V_3^{q'} = I
\]
as in \eqref{TheVs}, with $V_1$ of full spectrum, where $\Phi'$ is the Flip automorphism of $B$ defined by
\[
\Phi'(V_3) = V_3^{-1}, \qquad \Phi'(V_1) = V_1^{-1}.
\]

\medskip

\begin{proof} We begin with the canonical epimorphism of the rational rotation algebra onto $B,$
\[
\pi: A_{p'/q'} \to C^*(V_1,V_3), \qquad \pi(U') = V_1,\qquad \pi(V') = V_3
\]
where $U' = U_{p'/q'}, \ V' = V_{p'/q'}$ satisfy $V'U' = e(\tfrac{p'}{q'}) U'V'$. This surjection intertwines the two Flips
\[
\pi \Phi = \Phi'\pi 
\]
which is easy to verify, where $\Phi(U')={U'}^{-1}, \Phi(V') = {V'}^{-1}$ is the Flip on $A_{p'/q'}$.

\medskip

Fix a $\Phi'$-trace $\psi$ on $B$. Since $\psi \pi$ is a $\Phi$-trace on $A_{p'/q'},$ it is a linear combination of the four basic unbounded $\Phi$-traces defined on the basic unitaries ${U'\,}^m{V'\,}^n$ by 
\begin{equation}
\phi_{ij}({U'\,}^m{V'\,}^n)\ =\ e(-\tfrac{p'mn}{2q'})\,\updelta_2^{m-i} \updelta_2^{n-j}
\end{equation}
($ij = 00, 01,10,11$). Thus
\[
\psi \pi = a \phi_{00} + b\phi_{01} + c\phi_{10} + d\phi_{11}
\]
for some constants $a,b,c,d$. Evaluating this gives on the basic unitaries,
\[
\psi( V_1^{m} V_3^n ) = \psi \pi ({U'}^m{V'}^n) 
= a \phi_{00}({U'}^m{V'}^n) + b\phi_{01}({U'}^m{V'}^n) + c\phi_{10}({U'}^m{V'}^n) 
+ d\phi_{11}({U'}^m{V'}^n) 
\]
\[
= e(-\tfrac{p'mn}{2q'})
\Big( a\updelta_2^{m} \updelta_2^{n} + b\updelta_2^{m} \updelta_2^{n-1}
+ c\updelta_2^{m-1} \updelta_2^{n} + d\updelta_2^{m-1} \updelta_2^{n-1}
\Big)
\]
from $e(\tfrac{p'mn}{q'}) V_1^{m} V_3^n = V_3^n V_1^{m},$ this gives
\begin{equation}\label{psiabcd}
\psi( V_3^{n} V_1^m ) = e(\tfrac{p'mn}{2q'}) 
\Big( a\updelta_2^{m} \updelta_2^{n} + b\updelta_2^{m} \updelta_2^{n-1}
+ c\updelta_2^{m-1} \updelta_2^{n} + d\updelta_2^{m-1} \updelta_2^{n-1}
\Big).
\end{equation}
Since this expression should be invariant under the translation $n \to n+q'$ (as $V_3$ has order $q'$), we get
\begin{align*}
a\updelta_2^{m} \updelta_2^{n} &+ b\updelta_2^{m} \updelta_2^{n-1}
+ c\updelta_2^{m-1} \updelta_2^{n} + d\updelta_2^{m-1} \updelta_2^{n-1}
\\
&=
(-1)^{p'm} a\updelta_2^{m} \updelta_2^{n+q'} + (-1)^{p'm} b\updelta_2^{m} \updelta_2^{n+q'-1}
+ (-1)^{p'm} c\updelta_2^{m-1} \updelta_2^{n+q'} + (-1)^{p'm} d\updelta_2^{m-1} \updelta_2^{n+q'-1}
\end{align*}
or
\begin{align*}
a\updelta_2^{m} \updelta_2^{n} + b\updelta_2^{m} \updelta_2^{n-1}
&+ c\updelta_2^{m-1} \updelta_2^{n} + d\updelta_2^{m-1} \updelta_2^{n-1}
\\
& =
a\updelta_2^{m} \updelta_2^{n+q'} + b\updelta_2^{m} \updelta_2^{n+q'-1}
+ (-1)^{p'} c\updelta_2^{m-1} \updelta_2^{n+q'} + (-1)^{p'} d\updelta_2^{m-1} \updelta_2^{n+q'-1}
\end{align*}
which must be satisfied for all four parities of $m,n$. Setting $m=n=0,$ it gives
\[
a =  a\updelta_2^{q'} + b\updelta_2^{q'-1} 
\]
and for $m=0, n=1$:
\[
b =  a \updelta_2^{q'-1} + b \updelta_2^{q'}
\]
for $m=1, n=0$:
\[
c = (-1)^{p'} c \updelta_2^{q'} + (-1)^{p'} d \updelta_2^{q'-1}
\] 
and for $m=n=1$:
\[
d =  (-1)^{p'} c \updelta_2^{q'-1} + (-1)^{p'} d \updelta_2^{q'}
\]

We consider separately two parity cases for $q'$. 

If $q'$ is even (so $p'$ is odd), these become $a =  a, \ b =  b, \ c = - c, \ d = - d,$
so $c=d=0,$ and equation \eqref{psiabcd} becomes
\[
\ \ \qquad \psi( V_3^{n} V_1^m ) = e(\tfrac{p'mn}{2q'}) 
\Big( a \updelta_2^{n} + b \updelta_2^{n-1} \Big) \updelta_2^{m},\qquad (q' \ \text{even})
\]
where $a,b$ are arbitrary scalars. Taking $a=b=1$ and $a=1,b=-1$, gives us the two basic unbounded traces
\[
\ \ \qquad 
\psi_{1} ( V_3^{n} V_1^m ) = e(\tfrac{p'mn}{2q'})  \updelta_2^{m}, \qquad 
\psi_{2} ( V_3^{n} V_1^m ) = e(\tfrac{p'mn}{2q'}) (-1)^{n} \updelta_2^{m} 
\qquad (q' \ \text{even}).
\]
These are easily verified to be well-defined under $n\to n+q',$ where $q'$ is even here, and that both are $\Phi'$-traces. 

\bigskip

If $q'$ is odd, we get $a = b, \  d =  (-1)^{p'} c,$ from which we can take $a$ and $c$ to be independent parameters, and equation \eqref{psiabcd} in this case becomes
\begin{align*}
\psi( V_3^{n} V_1^m ) &= e(\tfrac{p'mn}{2q'}) 
\Big( a\updelta_2^{m} \updelta_2^{n} + a\updelta_2^{m} \updelta_2^{n-1}
+ c\updelta_2^{m-1} \updelta_2^{n} + c (-1)^{p'} \updelta_2^{m-1} \updelta_2^{n-1}
\Big)
\\
& = a e(\tfrac{p'mn}{2q'}) ( \updelta_2^{n} + \updelta_2^{n-1})  \updelta_2^{m}
+ c e(\tfrac{p'mn}{2q'}) ( \updelta_2^{n} +  (-1)^{p'} \updelta_2^{n-1})  \updelta_2^{m-1} 
\\
& = a e(\tfrac{p'mn}{2q'})  \updelta_2^{m}
+ c e(\tfrac{p'mn}{2q'})  (-1)^{p'n}  \updelta_2^{m-1} 
\end{align*}
giving us the two basic unbounded traces in the odd $q'$ case
\[
\psi_1( V_3^{n} V_1^m ) = e(\tfrac{p'mn}{2q'})  \updelta_2^{m}, 	\qquad
\psi_2( V_3^{n} V_1^m ) = e(\tfrac{p'mn}{2q'}) (-1)^{p'n} \updelta_2^{m-1}.
\]
Combining the two parity cases for $q'$, we can write the two basic $\Phi'$-traces as in \eqref{Phiprimetraces} above. (Note that $\psi_2$ here agrees with the odd $q'$ case, and when $q'$ is even, $p'$ has to be odd so $(-1)^{p'n} = (-1)^{n}$ as in $\psi_2$ in the even case.) 
\end{proof}

Notice that by contrast with the unbounded $\Phi$-trace functionals $\phi_{jk}$ for the smooth rotation algebra, the $\Phi'$-traces $\psi_1,\psi_2$ are continuous linear functionals on the circle algebra $M_{q'}(C(\mathbb T))$. Thus, the ``unbounded traces" here turn out to be bounded.

\medskip

\subsection*{Acknowledgement}
This paper and \ccite{WaltersModular} were written at about the time the author retires. He is therefore most grateful to his home institution of 26 years, the University of Northern British Columbia, for many years of research and so much other support. The author expresses his nontrivial gratitude to the many referees who made helpful review reports over the years (including critical ones). Thank you.


\bigskip

\begin{thebibliography}{10}
 
\bibitem{BKR}
B. Blackadar, A. Kumjian, and M. R\o rdam, \emph{Approximately Central Matrix Units and the Structure of Noncommutative Tori}, K-theory {\bf6} (1992), 267--284.

\bibitem{FB}
F. P. Boca, \emph{Projections in rotation algebras and theta functions}, Comm. Math. Phys. {\bf202} (1999), 325--357.

\bibitem{BEEKa}
O.~Bratteli, G.~A.~Elliott, D.~E.~Evans, A.~Kishimoto, \emph{Non-commutative spheres I}, Internat. J. Math. {\bf2} (1990), no. 2, 139--166.

\bibitem{BK}
O. Bratteli, A. Kishimoto, \emph{Non-commutative spheres III: Irrational rotation}, Comm. Math. Phys. {\bf147} (1992), 605--624.

\bibitem{BW}
J. Buck and S. Walters, \emph{Connes-Chern characters of hexic and cubic modules}, J. Operator Theory {\bf57} (2007), 35--65.

\bibitem{ELPW}
S. Echterhoff, W. L\"uck, N. C. Phillips, S. Walters,
\emph{The structure of crossed products of irrational rotation algebras by 
finite subgroups of ${\mathrm SL}_2(\mathbb Z)$}, J. Reine Angew. Math. (Crelle's Journal), to appear, 43 pages.

\bibitem{GE1984}
G. Elliott, \emph{On the K-theory of the C*-algebra generated by a projective representation of a torsion-free discrete abelian group}, in Operator algebras and group representations (Pitman,
London, 1984),151--184.

\bibitem{GE}
G. Elliott, \emph{On the classification of C*-algebras of real rank zero}, 
J. Reine Angew. Math. (Crelle's Journal) {\bf443} (1993), 179--219.

\bibitem{EE}
G. Elliott and D. Evans, \emph{The structure of the irrational rotation C*-algebra},
Ann. Math. {\bf138} (1993), 477--501.

\bibitem{EL}
G. A. Elliott and Q. Lin, \emph{Cut-down method in the inductive limit decomposition of non-commutative tori}, J. London Math. Soc. (2) {\bf54} (1996), 121--134.


\bibitem{AK}
A. Kishimoto, \emph{Central sequence algebras of a purely infinite simple 
C*-algebra}, Canad. J. Math. {\bf 56}, No. 6 (2004), 1237--1258.


\bibitem{MRb}
M. Rieffel, \emph{Projective modules over higher-dimensional non-commutative tori}, Canad. J. Math {\bf40} (1988), 257--338.

\bibitem{Rordam}
M. Rordam, E. Stormer, \emph{Classification of Nuclear C*-Algebras. Entropy in Operator Algebras}, Springer-Verlag, Berlin, 2002.

\bibitem{SWa}
S. G. Walters, \emph{Projective modules over the non-commutative sphere}, J. London Math. Soc. {\bf51}, No. 2 (1995), 589--602.

\bibitem{SW-CMP}
S. G. Walters, \emph{Inductive limit automorphisms of the irrational rotation algebra}, Comm. Math. Phys. {\bf 171} (1995), 365--381.

\bibitem{SWChern}
S.~Walters, \emph{Chern characters of Fourier modules}, Canad. J. Math. {\bf 52} (2000), No. 3, 633--672.

\bibitem{SWcjm}
S. Walters, \emph{K-theory of non commutative spheres arising from the Fourier 
automorphism}, Canad. J. Math. {\bf53}, No. 3 (2001), 631--672.


\bibitem{SWcrelles}
S. Walters, \emph{The AF structure of non commutative toroidal 
$\mathbb Z/4\mathbb Z$ orbifolds}, 
J. Reine Angew. Math. (Crelle's Journal) {\bf568} (2004), 139--196. 

\bibitem{SWrokhlin}
S. Walters, \emph{The exact tracial Rokhlin property}, Houston J. Math. {\bf41} (2015), No. 1, 265--272.

\bibitem{SamHouston2018}
S. Walters, \emph{Semiflat orbifold projections}, Houston J. Math {\bf44} (2018), No. 2, 645--663. arXiv:1711.01016

\bibitem{WaltersModular} 
S. Walters, \emph{Modular Images Of Approximately Central Projections}, preprint (2020), 12 pages. arxiv.org/abs/2006.07728

\bibitem{WaltersKAC} 
S. Walters, \emph{K-theory of approximately central projections: Fourier case}, in progress, approx 36 pages. 

\end{thebibliography}
\end{document}